\documentclass{amsart}

\usepackage[foot]{amsaddr}

\usepackage{geometry}
\usepackage{amsmath}
\usepackage{amssymb}
\usepackage{mathrsfs}
\usepackage{MnSymbol}
\usepackage{graphicx}
\usepackage{float}
\usepackage{amsthm}
\usepackage{tikz-cd}
\usepackage{dynkin-diagrams}
\usepackage{hyperref}
\usepackage{tikz}
\usepackage[utf8]{inputenc}
\usepackage{multirow}

\title[Legendrian Embeddings of Rational Homogeneous Spaces into Nilpotent Orbits]{Classification of Equivariant Legendrian Embeddings of Rational Homogeneous Spaces into Nilpotent Orbits}
\author{Minseong Kwon}
\address{Morningside Center of Mathematics, Academy of Mathematics and Systems Science, Chinese Academy of Sciences, Beijing, 100190, China}
\email{minseong@amss.ac.cn; 0009-0009-6567-3090 (ORCID)}


\date{September 11, 2025; January 28, 2026 (v2)}

\newcommand{\rank}{\textrm{rank}}

\newcommand{\codim}{\textrm{codim}}

\newcommand{\pdiff}[2]{\frac{\partial #1}{\partial #2}}

\newtheorem{theorem}{Theorem}[section]
\newtheorem{coro}[theorem]{Corollary}
\newtheorem{lemma}[theorem]{Lemma}
\newtheorem{proposition}[theorem]{Proposition}

\newtheorem*{theorem*}{Theorem}
\newtheorem*{maintheorem*}{Main Theorem}
\newtheorem*{coro*}{Corollary}
\newtheorem*{lemma*}{Lemma}
\newtheorem*{prop*}{Proposition}
\newtheorem*{claim*}{Claim}

\theoremstyle{definition}
\newtheorem{defn}[theorem]{Definition}

\newtheorem{prob}{Problem}
\newtheorem*{prob*}{Problem}

\newtheorem{exa}[theorem]{Example}
\newtheorem*{exa*}{Example}

\newtheorem*{exer*}{Exercise}

\newtheorem*{acknowledgements*}{Acknowledgements}

\newtheorem{rmk}[theorem]{Remark}
\newtheorem*{rmk*}{Remark}
\newtheorem*{assump*}{Assumption}
\newtheorem*{notation*}{Notations}

\def \CC {\mathbb{C}}

\def \HH {\mathbb{H}}

\def \OO {\mathbb{O}}
\def \PP {\mathbb{P}}
\def \QQ {\mathbb{Q}}
\def \RR {\mathbb{R}}
\def \SS {\mathbb{S}}

\def \ZZ {\mathbb{Z}}

\def \Lcal {\mathcal{L}}

\def \Ncal {\mathcal{N}}
\def \Ocal {\mathcal{O}}

\def \Xcal {\mathcal{X}}

\def \afr {\mathfrak{a}}
\def \bfr {\mathfrak{b}}

\def \gfr {\mathfrak{g}}
\def \hfr {\mathfrak{h}}
\def \ifr {\mathfrak{i}}

\def \kfr {\mathfrak{k}}
\def \lfr {\mathfrak{l}}
\def \mfr {\mathfrak{m}}
\def \nfr {\mathfrak{n}}
\def \ofr {\mathfrak{o}}
\def \pfr {\mathfrak{p}}

\def \sfr {\mathfrak{s}}
\def \tfr {\mathfrak{t}}
\def \ufr {\mathfrak{u}}

\def \hbar {\bar{h}}

\begin{document}

\begin{abstract}
    For a complex semi-simple Lie algebra, every nilpotent orbit in its projectivization comes with a complex contact structure.
    For each nilpotent orbit, we classify projective Legendrian subvarieties that are homogeneous under the actions of their stabilizers in the adjoint group.
    In particular, we present a classification of equivariant Legendrian embeddings of rational homogeneous spaces into adjoint varieties.
\end{abstract}

\maketitle

\tableofcontents

\subsection*{Statements and Declarations}
The author declares no competing interests.

\newpage

\section{Introduction}

We are working over $\CC$, the field of complex numbers.
For a semi-simple Lie algebra $\sfr$ and a nilpotent orbit $Z \subset \PP(\sfr)$ (that is, an adjoint orbit of nilpotent elements), it is well known that $Z$ admits a natural (complex) contact structure which is invariant under the adjoint action.
In this setup, we classify equivariant Legendrian embeddings of rational homogeneous spaces into $Z$.
More precisely, we give an answer to the following problem:

\begin{prob} \label{prob: nilpotent}
    Let $\sfr$ be a semi-simple Lie algebra, $S_{\text{ad}}$ its adjoint group, and $Z \subset \PP(\sfr)$ a nilpotent orbit.
    Classify projective Legendrian subvarieties $O$ of $Z$ that are homogeneous under the actions of the stabilizers $\text{Stab}_{S_{\text{ad}}}(O)$.
\end{prob}

Of particular interest is the case where $\sfr$ is simple and $Z$ is the adjoint variety, i.e., the highest weight orbit of the adjoint representation $\sfr$.
In fact, adjoint varieties can be characterized as rational homogeneous spaces admitting invariant contact structures, as shown by Boothby \cite{Boothby61}.
Furthermore, adjoint varieties are the only known examples of Fano contact manifolds, and it has been conjectured that every Fano contact manifold is isomorphic to an adjoint variety (the so-called LeBrun-Salamon conjecture \cite{LS94Strong, Beauville1998FanoContact}).
Problem \ref{prob: nilpotent} is answered in Theorem \ref{main thm: adjoint} (in the case where $Z$ is an adjoint variety) and Theorem \ref{main thm: semi simple} (in the case where $Z$ is a nilpotent orbit other than an adjoint variety).

\subsection{Preliminaries}

Before stating the main theorems, let us recall some terminologies from contact geometry.
From now on, every manifold is assumed to be a complex manifold, and every variety is assumed to be a complex algebraic variety.
For a manifold $Z$ with $\dim Z > 1$, a holomorphic hyperplane subbundle $D \subset TZ$ is called a \emph{contact structure} if the bundle morphism $D \wedge D \rightarrow TZ / D$ defined by the Lie bracket of vector fields is everywhere non-degenerate.
In particular, for each $z\in Z$, the fiber $D_{z}$ is a conformal symplectic vector space.
For a submanifold $X \subset Z$, we say that $X$ is \emph{Legendrian} if its tangent spaces are Lagrangian (i.e., maximal isotropic) subspaces of the fibers of $D$.
If furthermore $Z$ is an algebraic variety and $X$ is a smooth subvariety that is Legendrian, we simply say that $X$ is a \emph{Legendrian subvariety} of $Z$.

Remark that Problem \ref{prob: nilpotent} is already solved in the case where $Z$ is the odd-dimensional projective space $\PP^{2n+1}$, i.e., the adjoint variety for a simple Lie algebra $C_{n+1} (=\mathfrak{sp}(2n+2))$.
Indeed, every projective Legendrian subvariety $O$ of $\PP^{2n+1}$ as in Problem \ref{prob: nilpotent} can be constructed as the space of lines on an adjoint variety passing through a given point.
Such a Legendrian subvariety is called a \emph{subadjoint variety} in literature.
See \cite[Table 1]{Buczynski2006LegendrianSubvarieties} for a list of all subadjoint varieties, \cite[\S A.1.3]{Buczynski2008AlgebraicLegendrian} for a history of the classification, and \cite[Theorem 11]{LandsbergManivel2007LegendrianVarieties} and \cite{Buczynski2006LegendrianSubvarieties} for generalizations of the classification in non-equivariant settings.

Even in the case where $Z$ is not a projective space, there is a known recipe to construct projective Legendrian subvarieties as in Problem \ref{prob: nilpotent}, using symmetric subalgebras.
Here, a subalgebra $\lfr$ of a semi-simple Lie algebra $\sfr$ is called \emph{symmetric} if there is a nontrivial Lie algebra involution $\theta: \sfr \rightarrow \sfr$ such that $\lfr = \sfr_{+1} := \{x \in \sfr: \theta(x) = x\}$.
Indeed, if we put $\sfr_{- 1} := \{x \in \sfr : \theta(x) = - x\}$, then for each nontrivial irreducible $\lfr$-subrepresentation $V \subset \sfr_{-1}$ and its highest weight orbit $O_{V} \subset \PP(V)$, $Z:= S_{\text{ad}} \cdot O_{V}$ is a nilpotent orbit and $O_{V}$ is a Legendrian subvariety of $Z$ (see \cite[Proposition 4.31]{BKP}).
Furthermore, using the well-known classification of symmetric subalgebras, one can easily obtain a complete list of such Legendrian subvarieties (see Propositions \ref{prop: Legendrian for IHSS} and \ref{prop: highest weight orbits for symm iso irre pair}).
Note, however, that not every equivariant Legendrian embedding arises in this way.
For example, the only subadjoint varieties arising from symmetric subalgebras are linear subspaces $\PP^{n} \subset \PP^{2n+1}$.

\subsection{Main results and outline}

Our main results are the following two theorems, giving a complete answer to Problem \ref{prob: nilpotent}.
The first main theorem concerns the case where $Z$ is an adjoint variety.

\begin{theorem} \label{main thm: adjoint}
    Let $\sfr$ be a simple Lie algebra, $S_{\text{ad}}$ its adjoint group, and $Z \subset \PP(\sfr)$ the adjoint variety.
    Assume that $O$ is a projective Legendrian subvariety of $Z$.
    If $O$ is homogeneous under the $\text{Stab}_{S_{\text{ad}}}(O)$-action, then $O$ is the highest weight $\lfr$-orbit contained in $\PP(V)$ for some reductive subalgebra $\lfr$ of $\sfr$ and an irreducible $\lfr$-subrepresentation $V$ of $\sfr$ belonging to the following list:
    \begin{enumerate}
        \item $\lfr = \sfr_{+1}$ and $V$ is any $\sfr_{+1}$-subrepresentation of $\sfr_{-1}$, where $\sfr_{\pm1}:=\{x \in \sfr: \theta(x) = \pm 1\}$ for a Lie algebra involution $\theta: \sfr \rightarrow \sfr$ such that $\sfr_{+1}$ is not isomorphic to the following symmetric subalgebras:
        \begin{enumerate}
            \item $C_{p} \oplus C_{l-p}$ in $\sfr = C_{l}$ ($1 \le p \le l-1$).
            
            \item $C_{l}$ in $\sfr = A_{2l-1}$ ($l \ge 2$).
            
            \item $\mathfrak{so}(l)$ in $\sfr = \mathfrak{so}(l+1)$ ($l \ge 2$).
            
            \item $B_{4}$ in $\sfr = F_{4}$.
            
            \item $F_{4}$ in $\sfr = E_{6}$; or
        \end{enumerate}
        
        \item $\lfr$ is a non-symmetric reductive subalgebra of $\sfr$ and $V$ is an irreducible $\lfr$-subrepresentation of $\sfr$ with highest weight $\rho$ given as follows:
        \begin{enumerate}
        \item $\sfr = C_{2}$, $\lfr = A_{1}$, and $\rho ={\dynkin[labels={6}] A{x}}$.
        \item $\sfr = C_{7}$, $\lfr = C_{3}$, and $\rho={\dynkin[labels={,,2}] C{**x}}$.
        \item $\sfr = C_{10}$, $\lfr = A_{5}$, and $\rho={\dynkin[labels={,,2}] A{**x**}}$.
        \item $\sfr = C_{16}$, $\lfr = D_{6}$, and $\rho={\dynkin[labels={,,,,2}] D{****x*}}$.
        \item $\sfr = C_{28}$, $\lfr = E_{7}$, and $\rho={\dynkin[backwards, upside down, labels={,,,,,,2}] E{******x}}$.
        \item $\sfr = C_{l}$ ($l \ge 3$), $\lfr = A_{1} \oplus \mathfrak{so}(l)$, and $\rho$ is associated to the tensor product of the standard representations of $A_{1}(=\mathfrak{sl}(2))$ and of $\mathfrak{so}(l)$.

        \item $\sfr = A_{2l-1}$ ($l \ge 3$), $\lfr = A_{l-1} \oplus A_{1}$, and $\rho={\dynkin[labels={1,,,1}] A{x*.*x} \otimes \dynkin[labels={2}] A{x}}$.
        \item $\sfr = A_{15}$, $\lfr = D_{5}$, and $\rho={\dynkin[labels={,,,1,1}] D{***xx}}$.
        \item $\sfr = A_{9}$, $\lfr = A_{4}$, and $\rho={\dynkin[labels={,1,1}] A{*xx*}}$.

        \item $\sfr = D_{8}$, $\lfr = B_{4}$, and $\rho={\dynkin[labels={,,1}] B{**x*}}$.
        \item $\sfr = D_{2l}$ ($l \ge 3$), $\lfr = C_{l} \oplus A_{1}$, and $\rho={\dynkin[labels={,1}] C{*x*.**} \otimes \dynkin[labels={2}] A{x}}$.
        \item $\sfr = E_{7}$, $\lfr = F_{4} \oplus A_{1}$, and $\rho={\dynkin[backwards, labels={,,,1}] F{***x} \otimes \dynkin[labels={2}] A{x}}$.
        \end{enumerate}
        In each case of (2.a--l), $\sfr = \lfr \oplus V$ as $\lfr$-representations and $V = \{x \in \sfr : b_{\sfr}(x,\,\lfr) = 0\}$ for the Killing form $b_{\sfr}$ of $\sfr$.
    \end{enumerate}
\end{theorem}

The list of Legendrian subvarieties obtained by Theorem \ref{main thm: adjoint}(1) is given in Propositions \ref{prop: Legendrian for IHSS} and \ref{prop: highest weight orbits for symm iso irre pair}.
In Theorem \ref{main thm: adjoint}(2), the items (2.a--f) (i.e., the cases where $\sfr$ is of type $C$) are corresponding to the non-linear subadjoint varieties.
For a description of the embedding of $\lfr$ into $\sfr$ given in (2), see Remark \ref{rmk:notation tables}.

Other nilpotent orbits are considered in our second main theorem.
In the following theorem, we use a labeling for nilpotent orbits in projectivized simple Lie algebras, explained in Section \ref{section: Contact Structure of Adjoint Variety}.

\begin{theorem} \label{main thm: semi simple}
    Let $\tilde{\sfr}$ be a semi-simple Lie algebra, $Z \subset \PP(\tilde{\sfr})$ a nilpotent orbit, and $\sfr$ the smallest ideal of $\tilde{\sfr}$ such that $Z \subset \PP(\sfr)$.
    Denote by $\tilde{S}_{\text{ad}}$ and $S_{\text{ad}}$ the adjoint groups of $\tilde{\sfr}$ and $\sfr$, respectively.
    Assume that $O$ is a projective Legendrian subvariety of $Z$.
    If $O$ is homogeneous under the $\text{Stab}_{\tilde{S}_{\text{ad}}}(O)$-action, then $O$ is $\text{Stab}_{S_{\text{ad}}}(O)$-homogeneous.
    Moreover, one of the following holds:
    \begin{enumerate}
        \item $\sfr$ is simple, $Z$ is the adjoint variety for $\sfr$, and $O$ can be obtained by Theorem \ref{main thm: adjoint}; or
        \item $O$ is the highest weight $\lfr$-orbit contained in $\PP(V)$ for some reductive subalgebra $\lfr$ of $\sfr$ and an irreducible $\lfr$-subrepresentation $V$ of $\sfr$ with highest weight $\rho$ given as follows:
        \begin{enumerate}
            \item $\sfr = \lfr' \oplus \lfr'$, $\lfr=\text{diag}(\lfr')$, and $\rho$ is the highest root of $\lfr'$ for a simple Lie algebra $\lfr'$. In this case, $Z = \PP (\Ocal_{\text{min}} \oplus \Ocal_{\text{min}})$ where $\Ocal_{\text{min}} \subset \lfr'$ is the adjoint orbit of a highest root vector of $\lfr'$.
            
            \item $\sfr = C_{l}$, $\lfr = C_{p} \oplus C_{l-p}$ ($1 \le p \le l-1$), and $\rho = {\dynkin[backwards, labels={1}] C{x*.**} \otimes \dynkin[labels={1}] C{x*.**}}$. In this case, $Z= Z_{\text{short}}$ in $\PP(\sfr)$. 
            
            \item $\sfr = A_{2l-1}$, $\lfr = C_{l}$ ($l \ge 2$), and $\rho = {\dynkin[labels={,1}] C{*x.**}}$. In this case, $Z= Z_{[2^{2},\, 1^{2l-4}]}$ in $\PP(\sfr)$.
            
            \item $\sfr = \mathfrak{so}(l+1)$, $\lfr = \mathfrak{so}(l)$ ($l \ge 2$), and $\rho$ is associated to the standard representation of $\mathfrak{so}(l)$.
            In this case, $Z= Z_{[3,\,1^{l-2}]} (= Z_{\text{short}} \text{ if $l$ is even})$ in $\PP(\sfr)$.
            
            \item $\sfr = F_{4}$, $\lfr = B_{4}$, and $\rho = {\dynkin[labels={,,,1}] B{***x}}$. In this case, $Z= Z_{\text{short}}$ in $\PP(\sfr)$.
            
            \item $\sfr = E_{6}$, $\lfr = F_{4}$, and $\rho = {\dynkin[backwards, labels={,,,1}] F{***x}}$. In this case, $Z= Z_{2A_{1}}$ in $\PP(\sfr)$.
            
            \item $\sfr = B_{3}$, $\lfr = G_{2}$, and $\rho = {\dynkin[labels={1,}] G{x*}}$. In this case, $Z=Z_{[3,\,2^{2}]}$ in $\PP(\sfr)$.
        \end{enumerate}
        In each case of (2.a--g), $Z$ is quasi-projective but not projective, $\sfr = \lfr \oplus V$ as $\lfr$-representations, and $V = \{x \in \sfr : b_{\sfr}(x,\,\lfr) = 0\}$ for the Killing form $b_{\sfr}$ of $\sfr$.
        Moreover, $\lfr$ is a symmetric subalgebra of $\sfr$ in (2.a--f), but not in (2.g).
    \end{enumerate}
\end{theorem}

To prove Theorems \ref{main thm: adjoint}--\ref{main thm: semi simple}, basic properties of nilpotent orbits as contact manifolds are reviewed in Section \ref{section: Contact Structure of Adjoint Variety}.

In Section \ref{section: Classification of Homogeneous Legendrian Subvarieties}, we explain the result of Merkulov \cite{Merkulov1997ExistenceGeometry}, which is an important tool to prove Theorems \ref{main thm: adjoint}--\ref{main thm: semi simple}.
In \cite{Merkulov1997ExistenceGeometry}, Merkulov shows that given a compact Legendrian submanifold $O$ of a contact manifold $Z$, there exists a complex manifold, called a Legendre moduli space, parametrizing a maximal family of Legendrian deformations of $O$ in $Z$, provided that a certain cohomology vanishes.
Moreover, Merkulov describes the tangent space of the Legendre moduli space in terms of the Kodaira map (introduced by Kodaira \cite{Kodaira1962TheoremCompleteness}).
In the setting of Theorems \ref{main thm: adjoint}--\ref{main thm: semi simple}, we show that a coset variety $S_{\text{ad}}/\text{Stab}_{S_{\text{ad}}}(O)$ is a Legendre moduli space associated to $O \subset Z$.
Moreover, its tangent space $\sfr/\ofr$ is an irreducible $\ofr$-representation where $\ofr$ is the Lie algebra of $\text{Stab}_{S_{\text{ad}}}(O)$ (Theorem \ref{thm: iso irre var as Legendrian moduli}).

This approach reduces the proof of Theorems \ref{main thm: adjoint}--\ref{main thm: semi simple} to the classification of subalgebras $\ofr$ of $\sfr$ such that $\sfr/\ofr$ is an irreducible $\ofr$-representation.
Since this condition implies that $\ofr$ is a maximal subalgebra, there are two possibilities: $\ofr$ is either parabolic, or reductive (cf. \cite[Corollaire~1, \S10, Ch.~VIII]{Bourbaki}; see also \cite[Theorem, \S30.4, Ch.~X]{Humphreys}).
In the former case, $\ofr$ is a parabolic subalgebra corresponding to an irreducible Hermitian symmetric space, whose classification is well known.
In the latter case, the classification of $\ofr$ can be read off from the results of Manturov \cite{Manturov1961HomogeneousAsymmetric, Manturov1961RiemannianSpaces, Manturov1966HomogeneousRiemannian}, Wolf \cite{Wolf1968GeometryStructure, Wolf1984CorrectionGeometry}, and Kr\"{a}mer \cite{Kramer1975KlassifikationBestimmter}.
We summarize the classification in Section \ref{section: Isotropy Irreducible Varieties}.

In Section \ref{section: Structure of Isotropy Representation}, we determine which subalgebras $\ofr$ of $\sfr$ such that the quotients $\sfr/\ofr$ are irreducible indeed define Legendrian subvarieties.
Namely, we show that each item in Theorems \ref{main thm: adjoint}--\ref{main thm: semi simple} defines a Legendrian subvariety of the prescribed nilpotent orbit.
In the process of the proof, we prove that Legendrian subvarieties in Theorem \ref{main thm: adjoint} are scheme-theoretic linear sections, with few exceptions belonging to Theorem \ref{main thm: adjoint}(1) (see Theorem \ref{thm: Legendrian then scheme-theoretic intersection} and Remark \ref{rmk: linear section}).

In Section \ref{section: proof of classification}, we complete the proof of Theorems \ref{main thm: adjoint}--\ref{main thm: semi simple}, and then present the following corollaries:

\begin{itemize}
    \item While all subadjoint varieties are Hermitian symmetric spaces, there are rational homogeneous spaces that are not Hermitian symmetric but admit equivariant Legendrian embeddings into nilpotent orbits (other than $\PP^{2n+1}$).
    We give a list of such rational homogeneous spaces in Corollary \ref{coro: not IHSS but Legendrian}.

    \item For a projective Legendrian subvariety $O \subset Z$ as in Problem \ref{prob: nilpotent}, it is possible that the whole $\text{Aut}(O)^{0}$-action does not extend to the $S_{\text{ad}}$-action, i.e., the restriction $\text{Stab}_{S_{\text{ad}}}(O)^{0} \rightarrow \text{Aut}(O)^{0}$ may not be surjective.
    (Here, the superscript $0$ stands for the identity component.)
    However, in such a case, we show that one can {\lq enlarge\rq} $Z$ to the adjoint variety $\tilde{Z}$ for a bigger simple Lie algebra $\tilde{\sfr} (> \sfr)$ so that $O$ is a Legendrian subvariety of $\tilde{Z}$, and $\text{Stab}_{\tilde{S}_{\text{ad}}}(O)^{0} \rightarrow \text{Aut}(O)^{0}$ is surjective for the adjoint group $\tilde{S}_{\text{ad}}$ of $\tilde{\sfr}$.
    See Corollary \ref{coro: extend aut} for details.
\end{itemize}

Finally, in Section \ref{section: Tables}, we give four tables used in the proof of Theorems \ref{main thm: adjoint}--\ref{main thm: semi simple}.

\subsection{Conventions}

We are working in the category of complex algebraic varieties, except for Section \ref{section: Classification of Homogeneous Legendrian Subvarieties} where we consider the holomorphic category.
Every variety is assumed to be an integral separated scheme of finite type over $\CC$.
For an algebraic group, an algebraic subgroup means a Zariski closed subgroup.
By a reductive algebra, we mean an algebraic linear Lie algebra that is the direct sum of its center and a semi-simple ideal.
For simple Lie algebras, our numbering of nodes of Dynkin diagrams is consistent with \cite[Table~1, \S2, Reference Chapter]{OnishchikVinberg1990LieGroups}.
From Section \ref{section: Contact Structure of Adjoint Variety}, we denote by $\sfr$ a semi-simple Lie algebra and by $b_{\sfr}$ its Killing form.
From Section \ref{section: Isotropy Irreducible Varieties}, we denote by $(\gfr,\,\hfr)$ a pair of reductive algebras $\hfr \subset \gfr$ satisfying the conditions in Definition \ref{defn: isotropy irre var} with $\dim \gfr > 1$.
In particular, $\gfr/\hfr$ is an irreducible $\hfr$-representation, and we write $\gfr = \hfr \oplus \mfr$ where $\mfr$ is the orthogonal complement of $\hfr$ in $\gfr$ with respect to the Killing form $b_{\gfr}$.
$O_{\mfr}$ means the highest weight orbit in $\PP(\mfr)$, and $Z_{\mfr} \subset \PP(\gfr)$ means a nilpotent orbit containing $O_{\mfr}$ (Proposition \ref{prop:Zm exists}).

\subsection*{Acknowledgments}
The author would like to thank his PhD advisor Jun-Muk Hwang for guiding this project.
The author is grateful to Jaros\l aw Buczy\'nski, Philippe Eyssidieux and Shin-young Kim for valuable discussions and for introducing enlightening references.
The author also would like to thank David Sykes for detailed comments on the draft of this paper.
This work is supported by the Institute for Basic Science (IBS-R032-D1).
The author is very grateful to the anonymous referee for the detailed report and helpful suggestions.

\section{Contact geometry of nilpotent orbits} \label{section: Contact Structure of Adjoint Variety}

In this section, we recall the notion of contact structures, and review contact geometry over nilpotent orbits.

\begin{defn} \label{defn: contact structure on any Z}
    Let $Z$ be a complex manifold of dimension $>1$, and $D \subset TZ$ a holomorphic vector subbundle.
    \begin{enumerate}
        \item The \emph{Levi tensor} $\text{Levi}^{D}$ is a bundle morphism defined as
        \[
            \text{Levi}^{D} : \bigwedge^{2} D \rightarrow TZ / D, \quad v \wedge w \mapsto [v,\, w] \mod D
        \]
        where $v$ and $w$ are local sections of $D$ and $[v,\,w]$ denotes the Lie bracket of vector fields.
        \item $D$ is called a \emph{contact structure} of $Z$ if $D \subset TZ$ is of corank 1 and $\text{Levi}^{D}_{z}$ is a non-degenerate 2-form on the fiber $D_{z}$ for every $z \in Z$.
        In this case, $Z$ is called a \emph{contact manifold}, and the quotient line bundle $\Lcal := TZ / D$ is called the \emph{contact line bundle}.

        \item A complex submanifold $X$ of a contact manifold $Z$ is called an \emph{integral submanifold} of the contact structure if $X$ is everywhere tangent to the contact structure, that is, $T_{x}X \subset D_{x}$ for any $x \in X$.
        If furthermore $\dim Z = 2 \dim X + 1$, we say that $X$ is a \emph{Legendrian} submanifold of $Z$.

        \item For a smooth variety equipped with a contact structure, a subvariety that is an integral/Legendrian submanifold is called an \emph{integral/Legendrian subvariety}.
    \end{enumerate}
\end{defn}

Note that for a contact structure $D$ of $Z$, its fiber $D_{z}$ at a point $z\in Z$ is equipped with a conformal symplectic structure induced by $\text{Levi}^{D}_{z} : \bigwedge^{2}D_{z} \rightarrow T_{z}Z/D_{z} (\simeq \CC)$.
Moreover, the tangent space of an integral submanifold at $z$ is an isotropic subspace of $D_{z}$.
Thus being a Legendrian submanifold means that its tangent space is a Lagrangian subspace of the contact structure at each point.

\begin{exa} \label{example: contact str of PTX}
    Let $Y$ be a complex manifold, and $Z:= \PP T^{*}Y$ its projectivized cotangent bundle.
    For $y \in Y$, each $z \in \PP T^{*}_{y}Y \subset Z$ corresponds to a hyperplane $\text{Ann}(z) \subset T_{y}Y$.
    If we define a hyperplane $\Theta_{z} \subset T_{z} Z$ as the preimage of $\text{Ann}(z)$ under the differential $T_{z}Z \rightarrow T_{y}Y$ of the natural projection, then $\Theta := \bigcup_{z \in Z} \Theta_{z}$ becomes a contact structure of $Z$.
    Moreover, it is well known that every Legendrian submanifold of $Z$ (with respect to $\Theta$) can be obtained as the projectivized conormal bundle $\PP N^{*}_{Y' / Y}$ of a complex submanifold $Y' \subset Y$.
\end{exa}

From now on, we always denote by $\sfr$ a semi-simple Lie algebra, by $b_{\sfr}$ its Killing form and by $S_{\text{ad}}$ its adjoint group (i.e., the identity component of $\text{Aut}(\sfr)$).

\begin{defn} \label{defn: adjoint variety}
    Let $\Ncal \subset \sfr$ be the cone of nilpotent elements.
    \begin{enumerate}
        \item An $S_{\text{ad}}$-orbit contained in $\PP (\Ncal)$ is called a \emph{nilpotent orbit} in $\PP(\sfr)$.
        \item If $\sfr$ is simple, the $S_{\text{ad}}$-orbit of a long root space in $\PP(\sfr)$ is called the \emph{adjoint variety} of $\sfr$, and denoted by $Z_{\text{long}}$.
    \end{enumerate}
\end{defn}

\begin{theorem}[{\cite[Remark 2.3]{Beauville1998FanoContact}}] \label{thm: boothby characterization of hom contact mfld}
    Let $Z \subset \PP(\sfr)$ be a nilpotent orbit.
    For each $[v] \in Z$ and the stabilizer $\nfr_{\sfr}(v) := \{w \in \sfr: [w,\,v] \in \CC \cdot v\}$, define a hyperplane $D_{[v]} := v^{\perp} / \nfr_{\sfr}(v) \subset \sfr / \nfr_{\sfr}(v) \simeq T_{[v]} Z$ where $v^{\perp} := \{x \in \sfr : b_{\sfr}(v,\,x) = 0\}$.
    The hyperplane subbundle of $TZ$ defined as $D = \bigcup_{[v] \in Z} D_{[v]}$ is an $S_{\text{ad}}$-invariant contact structure of $Z$.
\end{theorem}

\begin{rmk} \label{rmk: nilpotent orbit in s}
    \begin{enumerate}
        \item By a slight abuse of notation, we say that an $S_{\text{ad}}$-orbit in the cone of nilpotent orbits $\Ncal \subset \sfr$ is a \emph{nilpotent orbit} in $\sfr$.
        For a nilpotent orbit $Z \subset \PP(\sfr)$, its preimage $\Ocal \subset \sfr$ under the projection $\sfr \setminus \{0\} \rightarrow \PP(\sfr)$ is a single nilpotent orbit, since every $S_{\text{ad}}$-orbit in $\Ncal$ is $\CC^{\times}$-invariant.
        In this case, we write $Z = \PP (\Ocal)$.
        \item If $\sfr$ is simple, then the adjoint variety $Z_{\text{long}}$ is the highest weight orbit of the adjoint representation $\sfr$, i.e., the unique closed $S_{\text{ad}}$-orbit in $\PP(\sfr)$.
        Similarly, its preimage $\Ocal_{\text{min}} \subset \sfr$ under the projection $\sfr \setminus \{0\} \rightarrow \PP(\sfr)$ is the minimal nilpotent orbit, in the sense that $\Ocal_{\text{min}}$ is contained in the closure of every nonzero nilpotent orbit in $\sfr$.
        In this notation, we write $Z_{\text{long}} = \PP (\Ocal_{\text{min}})$.
    \end{enumerate}
\end{rmk}

\begin{proposition} \label{prop: contact line bundle of proj nil orbit}
    If $Z \subset \PP(\sfr)$ is a nilpotent orbit, then its contact line bundle is isomorphic to $\Ocal_{\PP(\sfr)}(1)|_{Z}$.
\end{proposition}
\begin{proof}
    Write $Z = S_{\text{ad}} / K$ for the stabilizer $K$ of a point, say $[v] \in Z$.
    The tangent bundle of $Z$ and its contact structure are given by $TZ \simeq S_{\text{ad}} \times_{K} (\sfr / \nfr_{\sfr}(v))$ and $D := S_{\text{ad}} \times_{K} (v^{\perp} / \nfr_{\sfr}(v))$, respectively.
    Thus the contact line bundle is given by $\Lcal := TZ/D \simeq S_{\text{ad}} \times_{K} (\sfr / v^{\perp})$.
    Observe that the Killing form induces a $K$-equivariant isomorphism $(\sfr/v^{\perp}) \simeq (\CC \cdot v)^{*}$, and hence
    \[
        \Lcal \simeq S_{\text{ad}} \times_{K} (\sfr / v^{\perp}) \simeq S_{\text{ad}} \times_{K} (\CC \cdot v)^{*} \simeq \Ocal_{\PP(\sfr)}(1)|_{Z}.
    \]
\end{proof}

\begin{proposition} \label{prop: no adjoint var is contact}
    Let $R \subset S_{\text{ad}}$ be an algebraic subgroup.
    Let $w \in \sfr$ be a nonzero nilpotent element, and $Z := S_{\text{ad}} \cdot [w] \subset \PP(\sfr)$.
    The $R$-orbit $R \cdot [w]$ is an integral subvariety of the contact structure of $Z$ if and only if $w$ is orthogonal to the Lie algebra of $R$ with respect to $b_{\sfr}$.
\end{proposition}
\begin{proof}
First, the orbit $R \cdot [w]$ is a (smooth) subvariety of $Z$ since $R$ is algebraic.
Recall that the contact structure of $Z$ at $[w]$ is given by $w^{\perp} / \nfr_{\sfr}(w) \subset \sfr / \nfr_{\sfr}(w) \simeq T_{[w]} Z$.
Under this identification, the tangent space of $R \cdot [w]$ is $T_{e}R \mod \nfr_{\sfr}(w) (= T_{e}R + \nfr_{\sfr}(w) / \nfr_{\sfr}(w))$, and so it is contained in the contact hyperplane $w^{\perp} / \nfr_{\sfr}(w)$ if and only if $T_{e}R \subset w^{\perp}$.
The statement follows since the contact structure of $Z$ is $S_{\text{ad}}$-invariant.
\end{proof}

Our notation for nilpotent orbits is as follows.
As before, $\sfr$ means a semi-simple Lie algebra, and $S_{\text{ad}}$ is its adjoint group.
When $\sfr$ is simple, the $S_{\text{ad}}$-orbit of long (short, respectively) root spaces are denoted by $Z_{\text{long}} \subset \PP(\sfr)$ ($Z_{\text{short}} \subset \PP(\sfr)$, respectively).
For other nilpotent orbits in a projectivized simple Lie algebra, we use the labeling of nilpotent orbits described in \cite{CollingwoodMcGovern1993NilpotentOrbits}.
Here is a brief explanation.

\begin{itemize}
    \item 
If $\sfr$ is a simple Lie algebra of classical type, then we consider its standard representation, and label each nilpotent orbit by the Jordan type of matrices lying in the orbit.
    That is, if $J_{d}$ is a ($d \times d$) elementary Jordan matrix
    \[
        J_{d} := \begin{pmatrix}
             0 & 1 & & & \\
             & 0 & 1 & & \\
             &&0&\ddots& \\
             & & &\ddots &1 \\
             & & & &0
        \end{pmatrix} \quad (d \ge 2), \quad \text{and} \quad J_{1} := (0),
    \]
    then $Z_{[d_{1}, \, \cdots, \, d_{k}]}$ denotes a nilpotent orbit in $\PP(\sfr)$ whose elements are represented by matrices conjugate to a Jordan matrix
    \[
        \begin{pmatrix}
             J_{d_{1}} & & & \\
             & J_{d_{2}} & & \\
             &&\ddots& \\
             & & & J_{d_{k}}
        \end{pmatrix}, \quad d_{1} \ge \cdots \ge d_{k} \ge 1.
    \]
    For example,
    \[
        Z_{\text{long}} = \left\{ \begin{array}{ll}
            Z_{[2,\,1^{r-1}]} & \text{if } \sfr = A_{r} = \mathfrak{sl}(r+1), \\
            Z_{[2,\, 1^{2r-2}]} & \text{if } \sfr = C_{r} = \mathfrak{sp}(2r), \\
            Z_{[2^{2},\, 1^{n-4}]} & \text{if } \sfr = \mathfrak{so}(n), \\
        \end{array} \right.
    \]
    and
    \[
        Z_{\text{short}} = \left\{ \begin{array}{ll}
            Z_{[2^{2},\, 1^{2r-4}]} & \text{if } \sfr = C_{r} = \mathfrak{sp}(2r), \\
            Z_{[3,\, 1^{2r-2}]} & \text{if } \sfr =B_{r} = \mathfrak{so}(2r+1). \\
        \end{array} \right.
    \]
    See \cite[\S5.4]{CollingwoodMcGovern1993NilpotentOrbits}.

\item If $\sfr$ is a simple Lie algebra of exceptional type, then we use the Bala-Carter classification, see \cite[\S 8.4]{CollingwoodMcGovern1993NilpotentOrbits}.
For example,
\[
    Z_{\text{long}} = Z_{A_{1}}, \quad \text{and} \quad Z_{\text{short}} = Z_{\tilde{A}_{1}}.
\]
As another example, when $\sfr = E_{6}$, for a nilpotent element $v \in \sfr$ such that the semi-simple part of a minimal Levi subalgebra containing $v$ is $A_{1} \oplus A_{1}$, then the nilpotent orbit $S_{\text{ad}} \cdot [v]$ is denoted by $Z_{2A_{1}}$.
\end{itemize}

\section{Legendre moduli space} \label{section: Classification of Homogeneous Legendrian Subvarieties}
In this section, we recall the notion of Legendre moduli spaces, introduced by Merkulov \cite{Merkulov1997ExistenceGeometry}, and reduce Theorems \ref{main thm: adjoint}--\ref{main thm: semi simple} to the classification of subalgebras $\ofr < \sfr$ such that $\sfr/\ofr$ is an irreducible $\ofr$-representation.

To be precise, first we recall the construction of the Kodaira map associated to an analytic family of compact submanifolds, introduced by Kodaira \cite{Kodaira1962TheoremCompleteness}.
To do this, in this section, we consider the setting of the holomorphic category.
Namely, every manifold is assumed to be a complex manifold, and every map between manifolds is holomorphic.
Suppose that $Z$ is a manifold and
\begin{center}
    \begin{tikzcd}[row sep=tiny]
        &\Xcal (\subset M \times Z) \arrow[ld, "p"] \arrow[rd, "q"]& \\
        M&& Z
    \end{tikzcd}
\end{center}
is a diagram of an analytic family of compact submanifolds of $Z$.
That is, $M$ is a connected manifold, $\Xcal$ is a submanifold of $M \times Z$, and the natural projection $p:\Xcal \rightarrow M$ is a proper submersion with connected fibers.
For each $t \in M$, put $\Xcal_{t} := q (p^{-1}(t))$, the submanifold corresponding to the point $t$, and choose a point $o \in M$.
Since $p$ is proper, there are finitely many coordinate neighborhoods $U_{i} \subset Z$, $i \in I$, say with coordinate functions $(w_{i}^{1},\,\ldots,\,w_{i}^{c}, \, z_{i}^{1},\, \ldots,\, z_{i}^{d})$, and a coordinate neighborhood $o \in U \subset M$ such that
\begin{itemize}
    \item $\Xcal_{o} \subset \bigcup_{i \in I}U_{i}$,
    \item for each $t \in U$ and $i \in I$, $\Xcal_{t} \cap U_{i}$ is defined by a system of equations $w_{i}^{\lambda} = \varphi_{i}^{\lambda}(t,\, z_{i}^{1},\, \ldots,\, z_{i}^{d})$, $\forall \lambda = 1,\,\ldots, \, c$, and
    \item for each $i \in I$ and $\lambda = 1,\,\ldots, \, c$, $\varphi_{i}^{\lambda}$ is a holomophic function on $U \times U_{i}$ satisfying $\varphi_{i}^{\lambda}|_{o \times U_{i}} = 0$.
\end{itemize}
Put $\varphi_{i} := (\varphi_{i}^{1}, \,\ldots,\, \varphi_{i}^{c})$, a vector-valued function.
For each tangent vector $\pdiff{}{t} \in T_{o} M$, the collection $\{\pdiff{\varphi_{i}}{t}\}_{i \in I}$ satisfies the cocycle condition for being a global section of the normal bundle $N_{\Xcal_{o}/Z}$.
Now the \emph{Kodaira map} is defined to be a $\CC$-linear map
\[
    \kappa : T_{o} M \rightarrow H^{0}(\Xcal_{o},\, N_{\Xcal_{o}/Z}), \quad \pdiff{}{t} \mapsto \left\{\pdiff{\varphi_{i}}{t}\right\}_{i \in I}.
\]

By the local nature of the Kodaira map, one can easily prove the following proposition:
\begin{proposition} \label{prop: kodaira map is natural}
    Let $\Xcal \rightarrow M$ and $\Xcal' \rightarrow M'$ be analytic families of compact submanifolds of manifolds $Z$ and $Z'$, respectively.
    Fix two points $o \in M$ and $o' \in M'$, and suppose that
    \begin{enumerate}
        \item there is a holomorphic map $f : M \rightarrow M'$ with $f(o) =o'$, and
        \item there exists a biholomorphism $F: U \rightarrow U'$ between open subsets $U \subset Z$ and $U' \subset Z'$ such that $\bigcup_{t \in M} \Xcal_{t} \subset U$, $\bigcup_{t' \in M'} \Xcal'_{t'} \subset U'$ and $F(\Xcal_{t}) = \Xcal'_{f(t)}$ for all $t \in M$.
    \end{enumerate}
    Then for the Kodaira maps $\kappa: T_{o}M \rightarrow H^{0}(\Xcal_{o}, \, N_{\Xcal_{o}/Z})$ and $\kappa' : T_{f(o)} M' \rightarrow H^{0}(\Xcal'_{f(o)},\, N_{\Xcal'_{f(o)}/Z'})$, we have a commutative diagram
    \begin{center}
        \begin{tikzcd}
            T_{o}M \arrow[r, "\kappa"] \arrow[d, "d_{o}f"] & H^{0}(\Xcal_{o}, \, N_{\Xcal_{o}/Z}) \arrow[d,"dF"] \\
            T_{f(o)} M' \arrow[r, "\kappa'"] & H^{0}(\Xcal'_{f(o)},\, N_{\Xcal'_{f(o)}/Z'})
        \end{tikzcd}
    \end{center}
    where the rightmost vertical map $dF$ is the isomorphism induced by the differential of the biholomorphism $F$.
\end{proposition}

Now Merkulov's result can be stated as follows:

\begin{theorem}[{\cite[Theorem 1.1]{Merkulov1997ExistenceGeometry}}] \label{thm: Merkulov result on Legendre moduli}
    Let $Z$ be a contact manifold with contact line bundle $\Lcal$.
    If $X$ is a compact Legendrian submanifold of $Z$ with $H^{1}(X,\,\Lcal|_{X}) = 0$, then there exists a manifold $M$ equipped with a diagram
    \begin{center}
        \begin{tikzcd}[row sep=tiny]
            & \Xcal (\subset M \times Z) \arrow[ld, "p"] \arrow[rd, "q"] & \\
            M && Z
        \end{tikzcd}
    \end{center}
    of an analytic family of compact Legendrian submanifolds of $Z$ containing $X$ such that the family is
    \begin{enumerate}
        \item complete, i.e., for each $t \in M$, the composition of the Kodaira map and the projection
        \[
            T_{t}M \xrightarrow{\kappa} H^{0}(\Xcal_{t},\, N_{\Xcal_{t}/Z}) \rightarrow H^{0}(\Xcal_{t},\, \Lcal|_{X_{t}})
        \]
        is an isomorphism; and
        \item maximal, i.e., for each $t\in M$, if there is another analytic family $M' \xleftarrow{p'} \Xcal' \xrightarrow{q'} Z$ of compact Legendrian submanifolds of $Z$ with $t' \in M'$ satisfying $\Xcal_{t} = \Xcal'_{t'}$, then there exist an open neighborhood $t' \in U' \subset M'$ and a holomorphic function $f : U' \rightarrow M$ such that $f(t') = t$ and $\Xcal_{f(t'')} = \Xcal'_{t''}$ for all $t'' \in U'$.
    \end{enumerate}
    The manifold $M$ is called a \emph{Legendre moduli space} associated to $X \subset Z$.
\end{theorem}

Next, let us apply Merkulov's result to our setting.
Recall that we denote by $\sfr$ a semi-simple Lie algebra.

\begin{defn} \label{defn: homogeneous legendrian}
Let $\lfr \subset \sfr$ be a reductive subalgebra.
A \emph{highest weight $\lfr$-orbit} in $\PP(\sfr)$ means the highest weight orbit in $\PP(V)$ for an irreducible $\lfr$-subrepresentation $V \subset \sfr$.
\end{defn}

\begin{rmk}
    If a projective subvariety $O$ of $\PP(\sfr)$ is homogeneous under the action of a connected algebraic subgroup of $S_{\text{ad}}$ with Lie algebra $\afr$, then it is a highest weight $\afr^{\text{Levi}}$-orbit where $\afr^{\text{Levi}}$ is a Levi subalgebra of $\afr$.
\end{rmk}

\begin{exa} \label{ex: legendrian in type A}
    If $\sfr = \mathfrak{sl}(n+2)$, $n \ge 1$, then $Z_{\text{long}} \simeq \PP T^{*} \PP^{n+1}$.
    The contact structures as an adjoint variety (Theorem \ref{thm: boothby characterization of hom contact mfld}) and as a projectivized cotangent bundle (Example \ref{example: contact str of PTX}) coincide.
    Recall that every Legendrian submanifold of $\PP T^{*} \PP^{n+1}$ is of the form $\PP N^{*}_{Y/\PP^{n+1}}$.
    The following are examples of $Y \subset \PP^{n+1}$ such that $\PP N^{*}_{Y/\PP^{n+1}}$ is homogeneous under the action of its stabilizer:
    \begin{enumerate}
        \item If $\PP^{d} \subset \PP^{n+1}$ is a linear subspace of dimension $d(\le n)$, then $\PP N^{*}_{\PP^{d}/\PP^{n+1}} (\simeq \PP^{d} \times \PP^{n-d})$ is homogeneous under the action of $\text{Stab}_{PGL(n+2)}(\PP^{d})$.
    Thus $\PP N^{*}_{\PP^{d}/\PP^{n+1}}$ is a highest weight ($D_{1}\oplus\mathfrak{sl}(d+1) \oplus \mathfrak{sl}(n+1-d)$)-orbit.
    (Here, $D_{1}$ denotes a 1-dimensional toral subalgebra.)
    \item If $\QQ^{n} \subset \PP^{n+1}$ is a smooth quadric hypersurface, then $\PP N^{*}_{\QQ^{n}/\PP^{n+1}} (\simeq \QQ^{n})$ is a highest weight $\mathfrak{so}(n+2)$-orbit.
    \end{enumerate}
    The subalgebras $D_{1}\oplus\mathfrak{sl}(d+1) \oplus \mathfrak{sl}(n+1-d)$ and $\mathfrak{so}(n+2)$ are symmetric subalgebras of $\mathfrak{sl}(n+2)$, and both $\PP N^{*}_{\PP^{d}/\PP^{n+1}}$ and $\PP N^{*}_{\QQ^{n}/\PP^{n+1}}$ can be obtained as in Theorem \ref{main thm: adjoint}(1).
    This is shown in Propositions \ref{prop: Legendrian for IHSS} (for $\PP^{d} \subset \PP^{n+1}$) and \ref{prop: highest weight orbits for symm iso irre pair} (for $\QQ^{n} \subset \PP^{n+1}$).
\end{exa}

In the following, we say that a group action is \emph{effective} if its kernel is trivial, or equivalently, the identity is the only element acting trivially.

\begin{theorem} \label{thm: iso irre var as Legendrian moduli}
    Let $Z \subset \PP(\sfr)$ be a nilpotent orbit.
    Assume that $O$ is a projective Legendrian subvariety of $Z$ that is homogeneous under the $\text{Stab}_{S_{\text{ad}}}(O)$-action.
    For a coset variety $M := S_{\text{ad}} / \text{Stab}_{S_{\text{ad}}}(O)$ and its base point $o := e \cdot \text{Stab}_{S_{\text{ad}}}(O)$, the diagram
    \[
        \begin{tikzcd}[row sep = tiny]
            & \Xcal:=\{ (g \cdot o,\, z) \in M \times Z \, :\, z \in g \cdot O \} (\simeq S_{\text{ad}} \times^{\text{Stab}_{S_{\text{ad}}}(O)}O) \arrow[ld] \arrow[rd] & \\
            M && Z
        \end{tikzcd}
    \]
    equipped with the natural projections defines a complete and maximal analytic family of compact Legendrian submanifolds of $Z$.
    Moreover, if there is no ideal $\ifr \subset \sfr$ such that $Z \subset \PP(\ifr)$, then for the Lie algebra $\ofr$ of $\text{Stab}_{S_{\text{ad}}}(O)$, we have the following:
        \begin{enumerate}
            \item $S_{\text{ad}}$ acts on $M$ effectively;
            \item $\sfr/\ofr$ is an irreducible $\ofr^{\text{Levi}}$-representation; and
            \item $O \subset \PP(\ofr^{\perp})$ where $\ofr^{\perp} := \{x \in \sfr : b_{\sfr}(x,\,\ofr) = 0\}$.
        \end{enumerate}
\end{theorem}
\begin{proof}
    Let $V \subset \sfr$ be the irreducible $\ofr^{\text{Levi}}$-subrepresentation $V\subset \sfr$ such that $O \subset \PP(V)$ is the highest weight orbit.
    Since $\Lcal|_{O} \simeq \Ocal_{\PP(V)}(1)|_{O}$ by Proposition \ref{prop: contact line bundle of proj nil orbit}, by the Bott-Borel-Weil theorem, $H^{0}(O,\, \Lcal|_{O}) \simeq V^{*}$ as $\ofr^{\text{Levi}}$-representations while $H^{q}(O,\, \Lcal|_{O}) = 0$, $\forall q \ge 1$.
    In particular, by Theorem \ref{thm: Merkulov result on Legendre moduli}, there exists a Legendre moduli space $M'$ associated to $O$.

    Next, consider the diagram in the statement.
    Since $O$ is a projective subvariety, $\text{Stab}_{S_{\text{ad}}}(O)$ is an algebraic subgroup of $S_{\text{ad}}$, and hence $M$ is also a variety and the $S_{\text{ad}}$-action on $M$ is algebraic.
    Moreover, with respect to the $S_{\text{ad}}$-action on $M \times Z$ (defined by $g \cdot (m,\, z) := (g \cdot m,\, g \cdot z)$), $\Xcal$ is a single orbit, and hence a smooth subvariety of $M \times Z$.
    Since the morphism $\Xcal \rightarrow M$ is a $S_{\text{ad}}$-homogeneous fiber bundle with fiber $\simeq O$, the diagram defines an analytic family of compact Legendrian submanifolds of $Z$.

    To prove that this family is complete and maximal, by homogeneity and by \cite[Lemma 2.2]{Merkulov1997ExistenceGeometry}, it is enough to show that the family is complete at some $x \in M$.
    To see this, for each $x \in M$, consider the composition
    \[
        T_{x}M \xrightarrow{\kappa_{x}} H^{0}(\Xcal_{x},\, N_{\Xcal_{x}/Z}) \xrightarrow{r_{x}} H^{0}(\Xcal_{x},\, \Lcal|_{\Xcal_{x}})
    \]
    where $\kappa_{x}$ is the Kodaira map for $M$ and $r_{x}$ is the restriction map.
    Observe that if $x = g \cdot o$, then $\Xcal_{x} = g \cdot O$, and hence $H^{0}(\Xcal_{x},\, \Lcal|_{\Xcal_{x}}) \simeq (Ad_{g}V)^{*}$ is an irreducible $Ad_{g}(\ofr^{\text{Levi}})$-representation.
    Since $\kappa_{x}$ is $\text{Stab}_{S_{\text{ad}}}(g \cdot O)$-equivariant by Proposition \ref{prop: kodaira map is natural}, we see that $r_{x} \circ \kappa_{x}$ is either zero or surjective.
    Furthermore, by homogeneity (and by Proposition \ref{prop: kodaira map is natural}), we have either $r_{x} \circ \kappa_{x} = 0$ for all $x \in M$, or $r_{x} \circ \kappa_{x}$ is surjective for all $x \in M$.
    On the other hand, since $M'$ is a Legendre moduli space, there exists an open neighborhood $o \in U \subset M$ and a map $f : U \rightarrow M'$ with $f(o) = [O]$.
    Since $\Xcal_{s} \not= \Xcal_{t}$ for $s\not=t \in M$, $f$ is injective.
    By Proposition \ref{prop: kodaira map is natural}, for each $x = g \cdot o \in U$, we have a commutative diagram
    \begin{center}
        \begin{tikzcd}
            T_{x}U = T_{x}M \arrow[rd, "\kappa_{x}"] \arrow[d, "d_{x}f"] & & \\
            T_{f(x)}M' \arrow[r, "\kappa'_{f(x)}"] & H^{0}(\Xcal_{x},\, N_{\Xcal_{x}/Z}) \arrow[r, "r_{x}"] & H^{0}(\Xcal_{x},\, \Lcal|_{\Xcal_{x}}) \simeq (Ad_{g}V)^{*}
        \end{tikzcd}
    \end{center}
    where $\kappa'_{f(x)}$ is the Kodaira map for $M'$.
    If $r_{x} \circ \kappa_{x} = 0$ for all $x \in M$, then since $r_{x} \circ \kappa'_{f(x)}$ is an isomorphism for all $x\in U$, $d_{x}f = 0$ for all $x \in U$, which means that $f$ is a constant map.
    However, since $f$ is injective, $M$ is a point, and hence $O = S_{\text{ad}} \cdot O = Z$, a contradiction.
    Therefore $r_{x} \circ \kappa_{x}$ is surjective for all $x \in M$.
    This implies that $d_{x}f$ is surjective for all $x\in U$, and hence $f$ is an isomorphism since $f$ is injective.
    Therefore $r_{x} \circ \kappa_{x}$ is an isomorphism, i.e., the family is complete at $x \in U$.

    Now assume that for every ideal $\ifr \subsetneq \sfr$, we have $Z \not\subset \PP(\ifr)$.
    We claim that the $S_{\text{ad}}$-action on $M$ is effective.
    If not, then
    \[
        K_{1} := \{g \in S_{\text{ad}} : g \cdot x = x, \, \forall x \in M\}
    \]
    is a nontrivial normal algebraic subgroup of $S_{\text{ad}}$.
    Since $S_{\text{ad}}$ is the adjoint group, $K_{1}$ is the product of some simple factors of $S_{\text{ad}}$, and we can write $S_{\text{ad}} = K_{1} \times K_{2}$, where $K_{2}$ is the product of remaining simple factors of $S_{\text{ad}}$.
    Now by the definition of $K_{1}$, we have $K_{1} \cdot o = o$, i.e., $K_{1} \subset \text{Stab}_{S_{\text{ad}}}(O)$.
    In particular, for every $[w] \in O$, the orbit $K_{1} \cdot [w]$ is contained in $O$, and hence an integral subvariety of $Z$.
    Thus by Proposition \ref{prop: no adjoint var is contact}, $w$ is contained in the orthogonal complement of $T_{e} K_{1}$, which is $T_{e} K_{2}$, an ideal of $\sfr$.
    It implies that $Z = S_{\text{ad}} \cdot [w] \subset \PP(T_{e}K_{2})$.
    By our assumption, we have $T_{e}K_{2} = \sfr$, and thus $K_{1}$ is a finite normal subgroup of $S_{\text{ad}}$.
    However, since $S_{\text{ad}}$ is the adjoint group, the only finite normal subgroup of $S_{\text{ad}}$ is the trivial subgroup.
    Hence $K_{1}$ is trivial, a contradiction.
    
    Therefore under the assumption, the $S_{\text{ad}}$-action on $M$ is effective.
    On the other hand, recall that we have shown that $\sfr/\ofr\simeq T_{o}M \simeq H^{0}(O,\, \Lcal|_{O}) \simeq V^{*}$ as $\ofr^{\text{Levi}}$-representations at the beginning of the proof, and hence $\ofr^{\text{Levi}}$ acts on $\sfr/\ofr$ irreducibly.
    Since $O$ is $\text{Stab}_{S_{\text{ad}}}(O)$-homogeneous, the last statement follows from Proposition \ref{prop: no adjoint var is contact}.
\end{proof}

Remark that Theorem \ref{thm: iso irre var as Legendrian moduli} may not hold for integral but not Legendrian subvarieties.
See Proposition \ref{ex: non-maximal family} for a counter-example.

\section{Classification of isotropy irreducible pairs} \label{section: Isotropy Irreducible Varieties}

As we have seen in Theorem \ref{thm: iso irre var as Legendrian moduli}, we need a classification of reductive subalgebras $\hfr \subset \gfr$ such that there exists a coset variety $G/H$ with $(T_{e}G,\,T_{e}H) = (\gfr,\,\hfr)$, with an effective $G$-action, and with an irreducible isotropy representation.
In this section, we present a classification of such pairs $(\gfr,\,\hfr)$, following \cite{Wolf1968GeometryStructure, Wolf1984CorrectionGeometry}.
In the proof of our main theorems (cf. Section~\ref{section: proof of classification}), $\gfr$ and $\hfr$ shall play the roles of $\sfr$ and $\ofr$ in Theorem~\ref{thm: iso irre var as Legendrian moduli}, respectively, in the case where $\ofr$ is reductive.
For simplicity, we introduce the following definitions:

\begin{defn} \label{defn: isotropy irre var}
    Let $\gfr$ be a reductive algebra and $\hfr$ a reductive subalgebra.
    \begin{enumerate}
        \item We say that the pair $(\gfr,\,\hfr)$ is an \emph{isotropy irreducible pair} if
        \begin{enumerate}
            \item the quotient representation $\gfr/\hfr$ is an irreducible $\hfr$-representation; and
            \item there exist a connected reductive group $G$ with $T_{e}G=\gfr$ and a reductive subgroup $H$ with $T_{e}H = \hfr$ such that the natural $G$-action on the coset variety $G/H$ is effective.
        \end{enumerate}
    In this case, the coset variety $G/H$ is called an \emph{isotropy irreducible variety of type $(\gfr,\,\hfr)$}.

        \item We say that an isotropy irreducible pair $(\gfr,\,\hfr)$ is \emph{symmetric} if $\hfr$ is a symmetric subalgebra of $\gfr$, that is, there is a Lie algebra involution $\theta: \gfr \rightarrow \gfr$ such that $\hfr$ is the $(+1)$-eigenspace of $\theta$.
    \end{enumerate}
    
\end{defn}

In the following, a \emph{Lie subgroup} of a real Lie group means a subgroup that is a closed real submanifold.

\begin{proposition} \label{prop:effective actions}
    Let $G_{\RR}$ be a compact connected real Lie group and $H_{\RR}$ a Lie subgroup of $G_{\RR}$.
    Let $G$ and $H$ be the complexifications of $G_{\RR}$ and $H_{\RR}$, respectively.
    The $G_{\RR}$-action on $G_{\RR}/H_{\RR}$ is effective if and only if the $G$-action on $G/H$ is effective.
\end{proposition}

\begin{proof}
    Recall that $H_{\RR} = G_{\RR} \cap H$ (\cite[Problem 24, \S5.2.5]{OnishchikVinberg1990LieGroups}), and so we can identify $G_{\RR} / H_{\RR}$ with a totally real submanifold $G_{\RR}\cdot (e \cdot H)$ of $G/H$.
    In fact, for any $x \in G_{\RR}\cdot (e \cdot H)$, the real vector space $T_{x}(G_{\RR}\cdot (e \cdot H))$ is a real form of $T_{x}(G/H)$, in the sense that $T_{x}(G/H) = T_{x}(G_{\RR}\cdot (e \cdot H)) \oplus \sqrt{-1} \cdot T_{x}(G_{\RR}\cdot (e \cdot H))$.

    Assume that the $G$-action on $G/H$ is effective.
    Let $g \in G_{\RR}$ be an element such that $g$ acts trivially on $G_{\RR} / H_{\RR}$.
    Since the $G$-action on $G/H$ is algebraic, the fixed-point-locus of $g$ on $G/H$ is Zariski closed.
    Since the fixed-point-locus contains a totally real submanifold $G_{\RR}/H_{\RR}$, it coincides with the whole variety $G/H$, and hence $g= e$.
    Thus the $G_{\RR}$-action on $G_{\RR}/H_{\RR}$ is effective.

    Assume that the $G_{\RR}$-action on $G_{\RR}/H_{\RR}$ is effective.
    If $K$ is the subgroup of $G$ defined by $K=\{g \in G : g \text{ acts on $G/H$ trivially}\}$, then $K$ is a normal algebraic subgroup of $G$ and $K \cap G_{\RR}=\{e\}$.
    Since $G$ is reductive, so is $K$.
    Thus $K$ is the complexification of its maximal compact subgroup, say $K_{\RR}$.
    Since $G_{\RR}$ is a maximal compact sbugroup of $G$, there exists $g \in G$ such that $g\cdot K_{\RR} \cdot g^{-1} \subset G_{\RR}$.
    Since $K$ is a normal subgroup, we see that
    \[
        g \cdot K_{\RR} \cdot g^{-1} \subset (g \cdot K \cdot g^{-1}) \cap G_{\RR} = K \cap G_{\RR} = \{e\}.
    \]
    Hence $K = \{e\}$, i.e., the $G$-action on $G/H$ is effective.
\end{proof}

Now a classification of isotropy irreducible pairs can be deduced from \cite{Wolf1968GeometryStructure, Wolf1984CorrectionGeometry}.
Indeed, in \cite{Wolf1968GeometryStructure, Wolf1984CorrectionGeometry}, there is a classification of Lie subgroups $H_{\RR}$ of a compact connected real Lie group $G_{\RR}$ satisfying the following conditions:
    \begin{enumerate}
        \item the $G_{\RR}$-action on $G_{\RR} / H_{\RR}$ is effective;
        \item if $\gfr_{\RR}$ and $\hfr_{\RR}$ are Lie algebras of $G_{\RR}$ and $H_{\RR}$, respectively, then $\gfr_{\RR}/\hfr_{\RR}$ is an $\hfr_{\RR}$-representation that is irreducible over $\RR$; and
        \item $G_{\RR}/H_{\RR}$ is simply connected.
    \end{enumerate}
By Proposition \ref{prop:effective actions}, our isotropy irreducible pairs $(\gfr,\,\hfr)$ are corresponding to the complexifications of $(\gfr_{\RR},\, \hfr_{\RR})$ in the classification in \cite{Wolf1968GeometryStructure} and \cite{Wolf1984CorrectionGeometry} such that $\gfr_{\RR} / \hfr_{\RR}$ is absolutely irreducible (i.e., irreducible over $\CC$).

\begin{theorem}[{\cite[Theorem 11.1]{Wolf1968GeometryStructure}, \cite{Wolf1984CorrectionGeometry}}] \label{thm: classification of isotropy irre pair}
    Let $(\gfr,\,\hfr)$ be an isotropy irreducible pair with $\dim \gfr > 0$.
    \begin{enumerate}
        \item If $(\gfr,\,\hfr)$ is not symmetric, then $(\gfr,\,\hfr)$ belongs to Table \ref{table: classification of g h}, and the highest weight $\rho$ of $\gfr/\hfr$ is given in the same table.
        In this case, $\gfr$ is simple, $\hfr$ is semi-simple and $\text{rank}(\hfr) < \text{rank}(\gfr)$.
        \item If $(\gfr,\,\hfr)$ is symmetric, then one of the following holds:
        \begin{enumerate}
            \item $\gfr = \mathfrak{so}(2)$ (1-dimensional reductive algebra) and $\hfr = 0$;
            \item $\gfr = \hfr' \oplus \hfr'$ and $\hfr = \text{diag}(\hfr')$ (i.e., $\hfr$ is the diagonal) for a simple Lie algebra $\hfr'$; and
            \item $\gfr$ is simple.
        \end{enumerate}
        In the case, $(\gfr,\,\hfr)$ other than $(\mathfrak{so}(2),\,0)$ belongs to Tables \ref{table: isotropy irreducible pairs of equal rank}--\ref{table: symmetric isotropy irreducible pairs of different rank with g simple}, and the highest weight $\rho$ of $\gfr/\hfr$ is given in the same tables.
    \end{enumerate}
\end{theorem}

Note that Theorem~\ref{thm: classification of isotropy irre pair} is equivalent to the classification of simply connected isotropy irreducible varieties.
Other isotropy irreducible varieties are their finite quotients.

\begin{exa} \label{example: non-symm isotropy irre embedding}
\begin{enumerate}
    \item
    Symmetric isotropy irreducible pairs are exactly the pairs $(T_{e}G,\, T_{e}H)$ arising from irreducible symmetric varieties $G/H$ not of Hermitian type.
    Here, $G/H$ is called \emph{symmetric} if there is an involution $\Theta : G \rightarrow G$ such that $G^{\Theta,\, 0} \subset H \subset G^{\Theta}$.
    A symmetric variety $G/H$ is called \emph{irreducible} if $G/H$ is not locally split into smaller symmetric varieties.
    Up to finite cover, every irreducible symmetric variety $G/H$ is one of the following types:
    \begin{itemize}
        \item (Torus) $G = \CC^{\times}$ and $H = \{e\}$.
        \item (Group) $G = H' \times H'$ and $H = \text{diag}(H')$ for a simple adjoint group $H'$.
        \item (Simple) $G$ is a simple adjoint group, and $H$ is semi-simple.
        \item (Hermitian) $G$ is a simple adjoint group, and $H$ is a Levi part of a parabolic subgroup $P$ such that $G/P$ is Hermitian symmetric and $\text{Aut}(G/P)^{0} = G$.
    \end{itemize}
    See \cite[\S8.10--8.11]{Wolf1984SpaceConstant}, \cite[\S26]{Timashev11book} and \cite[\S2.5]{BKP} for details.
    The highest weights of the isotropy representations $\gfr/\hfr$ can be found in \cite[Proof of Theorem 8.10.9 and (8.11.2)]{Wolf1984SpaceConstant} (when $\rank(\hfr) = \rank(\gfr)$) and \cite[(8.11.5)]{Wolf1984SpaceConstant} (when $\rank(\hfr) < \rank(\gfr)$).
    This information is summarized in Tables \ref{table: isotropy irreducible pairs of equal rank}--\ref{table: symmetric isotropy irreducible pairs of different rank with g simple}.

    \item
    The following are examples of an embedding $\hfr \hookrightarrow \gfr$ such that $(\gfr,\,\hfr)$ is a non-symmetric isotropy irreducible pair:
\begin{enumerate}
    \item The adjoint representation $\hfr \hookrightarrow \gfr:= \mathfrak{so}(\hfr)$ for simple $\hfr$ not of type $A$ (\cite[Corollary 10.2]{Wolf1968GeometryStructure}).
    In Table \ref{table: classification of g h}, No. $15_{n}$, $19_{n}$, $21_{n}$, 24, 26, 27, 28 and 29 correspond to the cases where $\hfr$ is $B_{n}$, $C_{n}$, $D_{n}$, $G_{2}$, $F_{4}$, $E_{6}$, $E_{7}$ and $E_{8}$, respectively.
    \item \label{ex: non symm pair from rational hom sp} Isotropy representations of some rational homogeneous spaces $L/P$ with $L$ simple, $P$ maximal parabolic and $\hfr$ the semi-simple part of the Lie algebra of $P$. More precisely, there is the smallest nonzero $P$-invariant subspace $T_{1}$ in $T_{e \cdot P} (L/P)$, and by comparing \cite[Proposition 2.6]{LandsbergManivel2003ProjectiveGeometry}, \cite[Theorem 11.1]{Wolf1968GeometryStructure} and \cite{Wolf1984CorrectionGeometry}, we have the following examples:
    \begin{enumerate}
        \item $\hfr \hookrightarrow \gfr := \mathfrak{sl}(T_{1})$ induced by an irreducible Hermitian symmetric space $L/P$ such that $\text{Aut}(L/P)^{0} = L$, neither a projective space nor a quadric. In this case, $T_{1} = T_{e \cdot P} (L/P)$.
        In Table \ref{table: classification of g h}, No. $1_{p,\,q}$, 2, 3, $4_{n}$ and $5_{n}$ are the cases where $L/P$ is $\text{Gr}(q,\, \CC^{p+q})$, $\OO\PP^{2}$ (the Cayley plane), $E_{7}/P_{1}$ (the $E_{7}$-Hermitian symmetric space), $\SS_{n}$ (the Spinor variety), and $\text{LG}(n,\, \CC^{2n})$ (the Lagrangian Grassmannian), respectively.
        
        \item $\hfr \hookrightarrow \gfr := \mathfrak{sp}(T_{1})$ induced by the adjoint variety $L/P$ for $L$ not of type $A$, $C$ (Definition \ref{defn: adjoint variety}).
        In Table \ref{table: classification of g h}, No. 6, 7, 8, 9, 10 and $11_{n}$ are the cases where the Lie algebra of $L$ is $G_{2}$, $F_{4}$, $E_{6}$, $E_{7}$, $E_{8}$ and $\mathfrak{so}(n+4)$, respectively.
        
        \item $\hfr \hookrightarrow \gfr := \mathfrak{so}(T_{1})$ induced by the isotropy representation of $\text{IG}(2,\, \CC^{2n+4})$ ($n \ge 3$), the isotropic Grassmannian of a symplectic vector space.
        Here, $\dim T_{1} = 4n$ (and $\codim T_{1} = 3$).
        This corresponds to No. $30_{n}$ in Table \ref{table: classification of g h}.
    \end{enumerate}
    Note that every non-symmetric isotropy pair $(\gfr, \,\hfr)$ with $\gfr$ of type $A$ or $C$ can be obtained in this way.
    
    \item Complexification of isotropy representations of certain Riemannian symmetric spaces. Indeed, these cover all non-symmetric $(\gfr,\,\hfr) \not=(B_{3},\,G_{2})$ with $\gfr$ classical, see \cite{WangZiller1993SymmetricSpaces} for a classification-free proof.
    For example, the previous examples in the item (\ref{ex: non symm pair from rational hom sp}) can be obtained by taking
    \begin{enumerate}
        \item the compact presentation of the Hermitian symmetric space $L/P$,
        \item the positive quaternionic-K\"{a}hler symmetric space of the same type with $L$, and
        \item the quaternionic projective space $\HH\PP^{n}$,
    \end{enumerate}
    respectively.
    See \cite[Tables 1--3]{WangZiller1993SymmetricSpaces}.
    \item The octonion representation $\hfr:= G_{2} \hookrightarrow B_{3} =: \gfr$.
    This pair is No. 23 in Table \ref{table: classification of g h}.
\end{enumerate}
\end{enumerate}
\end{exa}

From now on, we always denote by $(\gfr,\, \hfr)$ an isotropy irreducible pair with $\dim \gfr > 1$ (recall that this assumption implies that $\gfr$ is semi-simple by Theorem \ref{thm: classification of isotropy irre pair}).
Under this assumption, we fix our notation as follows.
First, let $G/H$ be an isotropy irreducible variety of type $(\gfr,\,\hfr)$, and $G_{\text{ad}} := G/Z(G)$ the adjoint group of $\gfr$.
For the Killing form $b_{\gfr}$ of $\gfr$, its restriction on $\hfr$ is non-degenerate (\cite[Theorem 2, \S4.1.1]{OnishchikVinberg1990LieGroups}), and so $\mfr := \{v \in \gfr : b_{\gfr}(v,\, \hfr) =0\}$ is a complementary subspace to $\hfr$ in $\gfr$.
That is, $\gfr = \hfr \oplus \mfr$ as $\hfr$-representations.
Denote by $O_{\mfr} \subset \PP(\mfr)$ the highest weight orbit.
That is, if $H^{0}$ is the identity component of $H$, then $O_{\mfr}$ is a unique closed $H^{0}$-orbit in $\PP(\mfr)$ (in fact, it is an $H$-orbit by Corollary \ref{coro: stabilizer of Om is the normalizer of h}).

Next, we choose a maximal toral subalgebra $\tfr_{H}$ of $\hfr$, and then the weight decompositions of $\hfr$ and $\mfr$ are given as follows:
\[
    \hfr = \tfr_{H} \oplus \bigoplus_{\alpha \in R_{\hfr}} \hfr_{\alpha}, \quad \mfr = \mfr_{0} \oplus \bigoplus_{w \in W} \mfr_{w} \quad (\mfr_{0} = 0 \text{ if and only if } \rank(\gfr) = \rank(\hfr)).
\]
Here, $R_{\hfr}$ is the set of roots of $\hfr$ and $W$ is the set of nonzero $\tfr_{H}$-weights of $\mfr$.
For $\alpha \in R_{\hfr}$ and $w \in W \cup \{0\}$, $E_{\alpha} \in \hfr_{\alpha} - \{0\}$ and $v_{w} \in \mfr_{w} - \{0\}$ mean a root vector of $\hfr$ and a weight vector of $\mfr$, respectively.
We also choose a Borel subalgebra $\bfr_{H} \subset \hfr$ containing $\tfr_{H}$.
The highest weight of $\mfr$ (with respect to $\bfr_{H}$) is denoted by $\rho \in W$ so that $O_{\mfr}$ is the orbit containing $[v_{\rho}] \in \PP(\mfr)$.
In fact, as we know $\rho$ explicitly (Theorem \ref{thm: classification of isotropy irre pair}, Tables \ref{table: classification of g h}--\ref{table: symmetric isotropy irreducible pairs of different rank with g simple}), we can describe $O_{\mfr}$ explicitly: for example, we list $\dim O_{\mfr}$ in Tables \ref{table: classification of g h}--\ref{table: symmetric isotropy irreducible pairs of different rank with g simple}.
For a simple Lie algebra $\hfr_{1}$, its simple roots and the highest root are denoted by $\alpha_{i}^{\hfr_{1}}$ and $\delta^{\hfr_{1}}$, respectively, indexed as in Section \ref{section: Tables}.
Note that the indexing is consistent with \cite{OnishchikVinberg1990LieGroups}.
If there is no ambiguity, we often omit the superscript $\hfr_{1}$.
Finally, we choose a maximal toral subalgebra $\tfr \le \tfr_{H} \oplus \mfr_{0}$ containing $\tfr_{H}$.
The set of roots of $\gfr$ is denoted by $R_{\gfr}$.

Now we prove basic properties of isotropy irreducible pairs.
The following corollary is a direct consequence of Theorem \ref{thm: classification of isotropy irre pair} and Tables \ref{table: classification of g h}--\ref{table: symmetric isotropy irreducible pairs of different rank with g simple}.

\begin{coro} \label{coro: corollary of table, characterization of B3 G2}
\begin{enumerate}
    \item $\rho$ is a root of $\hfr$ (that is, $\rho \in R_{\hfr}$) if and only if $\gfr$ is not simple or $(\gfr,\,\hfr)$ is one of $(B_{3},\, G_{2})$, $(A_{2l-1},\, C_{l})$ ($l \ge 2$), $(D_{p+1},\, B_{p})$ ($p \ge 2$) and $(E_{6},\, F_{4})$.
    \item If $\gfr$ is not simple, then $\rho = \delta$, the highest root of $\hfr$.
    \item If $\rho$ is a root of $\hfr$ and $\gfr$ is simple, $\rho$ is the dominant short root $\delta_{\text{short}}$ of $\hfr$.
\end{enumerate}
\end{coro}

\begin{proposition} \label{prop: maximalify of h and H}
    \begin{enumerate}
        \item $\hfr$ is a maximal subalgebra of $\gfr$.
        \item For the normalizer $N_{G_{\text{ad}}}(\hfr)$ of $\hfr$, the coset variety $G_{\text{ad}}/N_{G_{\text{ad}}}(\hfr)$ is an isotropy irreducible variety of type $(\gfr,\, \hfr)$.
        \item Under the quotient map $G \rightarrow G_{\text{ad}}$, the image of any algebraic subgroup of $G$ with Lie algebra $\hfr$ is contained in $N_{G_{\text{ad}}}(\hfr)$. In particular, there is a $G$-equivariant finite morphism $G/H \rightarrow G_{\text{ad}}/N_{G_{\text{ad}}}(\hfr)$.
    \end{enumerate}
\end{proposition}
\begin{proof}
    \begin{enumerate}
        \item It follows from the irreducibility of $\gfr / \hfr$.
        \item Since the Lie algebra of $N_{G_{\text{ad}}}(\hfr)$ contains $\hfr$, by the maximality of $\hfr$, it is either $\hfr$ or $\gfr$. By Theorem \ref{thm: classification of isotropy irre pair}, $\hfr < \gfr$ is not an ideal, and so its Lie algebra is $\hfr$.
        In fact, $\hfr$ does not contain any simple factor of $\gfr$, and so
        \[
            \overline{G}:=\{g \in G_{\text{ad}}: g \text{ acts trivially on }G_{\text{ad}}/N_{G_{\text{ad}}}(\hfr) \}
        \]
        is a finite subgroup, since $\overline{G}$ is a normal algebraic subgroup of $G_{\text{ad}}$ contained in $N_{G_{\text{ad}}}(\hfr)$.
        Since $G_{\text{ad}}$ is the adjoint group, $\overline{G} = \{e\}$, i.e., the $G$-action on $G_{\text{ad}}/N_{G_{\text{ad}}}(\hfr)$ is effective.
        \item It suffices to observe that every algebraic subgroup stabilizes its Lie algebra.
    \end{enumerate}
\end{proof}

\begin{coro} \label{coro: stabilizer of Om is the normalizer of h}
    The stabilizer $\text{Stab}_{G}(O_{\mfr})$ of $O_{\mfr} \subset \PP(\gfr)$ in $G$ is the preimage of $N_{G_{\text{ad}}}(\hfr)$ under the quotient map $G \rightarrow G_{\text{ad}}$.
\end{coro}
\begin{proof}
    Define $N:=N_{G_{\text{ad}}}(\hfr)$ and let $N^{0}$ be its identity component.
    By Proposition \ref{prop: maximalify of h and H}, it is enough to show that $N$ stabilizes $O_{\mfr}$.
    First, since $N$ stabilizes $\PP(\hfr)$ and $b_{\gfr}$ is $N$-invariant, $\PP(\mfr)$ is also $N$-stable, and so $g \cdot O_{\mfr} \subset \PP(\mfr)$ for $g \in N$.
    Since $N^{0} \cdot (g \cdot O_{\mfr}) = g \cdot (g^{-1}N^{0}g) \cdot O_{\mfr} = g \cdot (N^{0} \cdot O_{\mfr}) = g \cdot O_{\mfr}$, $g \cdot O_{\mfr}$ is a closed $N^{0}$-orbit, and hence it is equal to $O_{\mfr}$ by the irreducibility of $\mfr$.
\end{proof}

\begin{proposition} \label{prop:Zm exists}
    $O_{\mfr}$ is an integral subvariety of the contact structure of a nilpotent orbit $Z_{\mfr} := G_{\text{ad}} \cdot O_{\mfr} \subset \PP(\gfr)$.
\end{proposition}

\begin{proof}
    If $T_{H}$ is the maximal torus of $H$ with Lie algebra $\tfr_{H}$, then since $T_{H}$ acts nontrivially on $\mfr_{\rho}$, $v_{\rho}$ is a nilpotent element of $\gfr$ by \cite[Proposition 2.2]{Beauville1998FanoContact}.
    That is, $Z_{\mfr} (= G_{\text{ad}} \cdot [v_{\rho}])$ is a nilpotent orbit.
    Now the statement follows from Proposition \ref{prop: no adjoint var is contact}.
\end{proof}

From now on, we keep the notation of Proposition \ref{prop:Zm exists}: $Z_{\mfr}$ is a nilpotent orbit for $\gfr$, containing $O_{\mfr}$ as an integral subvariety.

\begin{rmk} \label{rmk: Legendrian when symm}
Proposition \ref{prop:Zm exists} is well known when $(\gfr,\,\hfr)$ is symmetric.
In fact, using \cite[Proposition 4.31]{BKP}, one can show that $O_{\mfr}$ is a Legendrian subvariety of $Z_{\mfr}$ if $(\gfr,\,\hfr)$ is symmetric.
This fact is recovered in Propositions \ref{prop:Om when g is not simple} and \ref{prop: highest weight orbits for symm iso irre pair} by using the classification of symmetric pairs.
\end{rmk}

If $O_{\mfr}$ is not Legendrian in $Z_{\mfr}$, then $M:=G_{\text{ad}}/\text{Stab}_{G_{\text{ad}}}(O_{\mfr})$ may not parametrize a maximal family of deformations into integral submanifolds in $Z_{\mfr}$ as in Theorem \ref{thm: iso irre var as Legendrian moduli}.
The following, which is not used in the rest of this paper, is an example:

\begin{proposition}\label{ex: non-maximal family}
    Let $(\gfr,\,\hfr)$ be the isotropy irreducible pair $(G_{2},\, A_{1})$ in No. 31, Table \ref{table: classification of g h}.
    Then there exists a maximal family of deformations of $O_{\mfr}$ as integral submanifolds of $Z_{\mfr}$ and the parameter space of the family is of dimension 23.
\end{proposition}

In particular, since $M = G_{\text{ad}}/\text{Stab}_{G_{\text{ad}}}(O_{\mfr})$ is of dimension $\dim G_{2} - \dim A_{1} = 11$ by Corollary~\ref{coro: stabilizer of Om is the normalizer of h}, we see that the family parametrized by $M$ is not maximal.

\begin{proof}[Proof of Proposition~\ref{ex: non-maximal family}]
    In Proposition~\ref{prop: rho is a highest root sp}, we shall show that $Z_{\mfr} = Z_{\text{long}}$, and for now let us assume this.
    Since $\mfr$ is the 10th symmetric power of the standard representation of $\mathfrak{sl}_{2}$, $O_{\mfr}$ is a smooth rational curve of degree 10 in $\PP(G_{2})$.
    For the contact line bundle $\Lcal$ on $Z_{\text{long}}$, we have $\Lcal|_{O_{\mfr}} \simeq \Ocal_{\PP^{1}}(10)$ by Proposition \ref{prop: contact line bundle of proj nil orbit}.
    If $D$ is the contact structure of $Z_{\text{long}}$, then since $\dim Z_{\text{long}} = 5$, $\rank(D) = 4$, $T O_{\mfr}$ is a line subbundle of $D|_{O_{\mfr}}$, and
    \[
        T O_{\mfr}^{\perp} := \{ v \in D_{x} : x\in O_{\mfr}, \ \text{Levi}^{D}_{x}(v,\, T_{x}O_{\mfr}) = 0 \}
    \]
    is a subbundle of $D|_{O_{\mfr}}$ of rank $3$, containing $T O_{\mfr}$.
    We claim that for the quotient bundle $S_{O_{\mfr}} := TO_{\mfr}^{\perp} / TO_{\mfr}$ (of rank 2), we have a short exact sequence
    \[
        0 \rightarrow \Ocal_{\PP^{1}}(4) \rightarrow S_{O_{\mfr}} \rightarrow \Ocal_{\PP^{1}}(6) \rightarrow 0.
    \]
    In fact, if the claim is true, then by \cite[Main Theorem, Ch. 4]{Ali03Thesis}, there exists a maximal family of deformations of $O_{\mfr}$ as integral submanifolds of $Z_{\text{long}}$ such that its parameter space is a complex manifold of dimension
    \[
        h^{0}(O_{\mfr},\, \Lcal|_{O_{\mfr}}) + h^{0}(O_{\mfr},\, S_{O_{\mfr}}) = 23,
    \]
    which proves the statement. To prove the claim, consider the weight decompositions
    \[
        \hfr = \hfr_{-1} \oplus \hfr_{0} \oplus \hfr_{1}, \quad \mfr = \bigoplus_{k=-5}^{5} \mfr_{k}
    \]
    where $\hfr_{k}$ and $\mfr_{k}$ are weight spaces of weight $k \cdot \alpha$ for the positive root $\alpha$ of $\hfr(=A_{1})$.
    Note that each weight space is of dimension 1, and so for the identity component $H$ of $\text{Stab}_{G_{\text{ad}}}(O_{\mfr})$, we have $O_{\mfr} = H \cdot [\mfr_{5}]$ and $Z_{\text{long}} = G_{\text{ad}} \cdot [\mfr_{5}]$.
    The Lie algebra of $P:=\text{Stab}_{G_{\text{ad}}}([\mfr_{5}])$ is given by
    \[
        l \oplus \hfr_{0} \oplus \hfr_{1} \oplus \bigoplus_{k \ge 0} \mfr_{k}
    \]
    for some 1-dimensional subspace $l < \hfr_{-1} \oplus \mfr_{-1}$ (with $l\not= \hfr_{-1}$).
    This shows that as $P$-representations,
    \[
        T_{[\mfr_{5}]} Z_{\text{long}} \simeq (\hfr \oplus \mfr) / \left( l \oplus \hfr_{0} \oplus \hfr_{1} \oplus \bigoplus_{k \ge 0} \mfr_{k} \right), \quad D_{[\mfr_{5}]} \simeq \left(\hfr \oplus \bigoplus_{k\ge -4}\mfr_{k}\right) / \left( l \oplus \hfr_{0} \oplus \hfr_{1} \oplus \bigoplus_{k \ge 0} \mfr_{k} \right).
    \]
    Since $T_{[\mfr_{5}]}O_{\mfr}$ is spanned by $\hfr$,
    \[
        (T O_{\mfr}^{\perp})_{[\mfr_{5}]} \simeq \left(\hfr \oplus \bigoplus_{k\ge -3}\mfr_{k}\right) / \left( l \oplus \hfr_{0} \oplus \hfr_{1} \oplus \bigoplus_{k \ge 0} \mfr_{k} \right),
    \]
    and hence
    \[
        (S_{O_{\mfr}})_{[\mfr_{5}]} \simeq \left(\hfr \oplus \bigoplus_{k\ge -3}\mfr_{k}\right) / \left(\hfr \oplus \bigoplus_{k\ge -1}\mfr_{k}\right) \simeq \bigoplus_{k\ge -3}\mfr_{k}/ \bigoplus_{k\ge -1}\mfr_{k}
    \]
    as $H \cap P$-representations.
    Now if we put $H$-homogeneous line bundles over $O_{\mfr} (\simeq H/H\cap P)$
    \[
        \Lcal_{-2} := H \times^{H \cap P} \left(\bigoplus_{k\ge -2}\mfr_{k} / \bigoplus_{k\ge -1}\mfr_{k}\right), \quad \Lcal_{-3} := H \times^{H \cap P} \left(\bigoplus_{k\ge -3}\mfr_{k}/\bigoplus_{k\ge -2}\mfr_{k}\right),
    \]
    then there is a short exact sequence of $H$-homogeneous vector bundles
    \[
        0 \rightarrow \Lcal_{-2} \rightarrow S_{O_{\mfr}} \rightarrow \Lcal_{-3} \rightarrow 0.
    \]
    Since $\alpha$ is 2 times the fundamental weight, by the Bott-Borel-Weil theorem, we see that $\Lcal_{-2} \simeq \Ocal_{\PP^{1}}(4)$ and $\Lcal_{-3} \simeq \Ocal_{\PP^{1}}(6)$.
    Therefore the claim follows.
\end{proof}

\section{Homogeneous Legendrian subvarieties arising from isotropy representations} \label{section: Structure of Isotropy Representation}

In this section, we show that each item in Theorems \ref{main thm: adjoint}--\ref{main thm: semi simple} indeed defines a Legendrian subvariety.
While it is well known that symmetric subalgebras define Legendrian subvarieties of some nilpotent orbits (see \cite[Proposition 4.31]{BKP}), for the sake of completeness, we also record its proof (see Propositions \ref{prop: Legendrian for IHSS}, \ref{prop:Om when g is not simple} and \ref{prop: highest weight orbits for symm iso irre pair}).

First, we consider highest weight $\lfr$-orbits where $(\sfr,\,\lfr)$ is a symmetric pair of Hermitian type (belonging to Theorem \ref{main thm: adjoint}(1)).
That is, $\lfr$ is a Levi subalgebra of a parabolic subalgebra corresponding to an irreducible Hermitian symmetric space.

\begin{proposition} \label{prop: Legendrian for IHSS}
    Assume that $\sfr$ is simple.
    Suppose that $\pfr < \sfr$ is a parabolic subalgebra such that for the associated parabolic subgroup $P < S_{\text{ad}}$, $S_{\text{ad}}/P$ is an irreducible Hermitian symmetric space and $\text{Aut}(S_{\text{ad}}/P)^{0} = S_{\text{ad}}$.
    There exists a unique closed $P$-orbit $O$ in $\PP(\sfr)$, and moreover, $O$ is a Legendrian subvariety of the adjoint variety $Z_{\text{long}} \subset \PP(\sfr)$.
    The list of $O$ is given in Table \ref{table: Legendrian asso to IHSS}.
\end{proposition}
\begin{proof}
    Since $\sfr$ is simple, $P$ has a unique closed orbit $O$ in $\PP(\sfr)$, which is the orbit containing the highest root space.
    Thus we have $Z_{\text{long}} = S_{\text{ad}} \cdot O$.
    Moreover, if we denote by $P^{\text{Levi}}$ a Levi subgroup of $P$, then $O$ is $P^{\text{Levi}}$-homogeneous.
    In fact, if we denote by $\pfr^{u}$ the unipotent radical of $\pfr$, then $O$ is the highest weight orbit of the irreducible $P^{\text{Levi}}$-representation $\pfr^{u}$ whose highest weight is the highest root $\delta$ of $\sfr$.
    Now $O$ can be read off from the well-known classification of irreducible Hermitian symmetric spaces: see Table \ref{table: Legendrian asso to IHSS}.
    In particular, we conclude that $2 \dim O + 1 = \dim Z_{\text{long}}$.
    Finally, observe that $O$ is an integral subvariety since $T_{e}P^{\text{Levi}}$ and $\pfr^{u}$ are orthogonal to each other and by Proposition \ref{prop: no adjoint var is contact}
\end{proof}

Next, we consider isotropy irreducible pairs $(\gfr,\,\hfr)$ with $\dim \gfr > 1$ (since the remaining cases in Theorems \ref{main thm: adjoint}--\ref{main thm: semi simple} are $O_{\mfr} \subset Z_{\mfr}$ for some isotropy irreducible pairs $(\gfr,\,\hfr)$).
More precisely, we determine when the integral subvariety $O_{\mfr} \subset Z_{\mfr}$ for $(\gfr,\,\hfr)$ is Legendrian.

The following is the case where $\gfr$ is not simple (corresponding to Theorem \ref{main thm: semi simple}(2.a)).

\begin{proposition} \label{prop:Om when g is not simple}
Assume that $\gfr$ is not simple, i.e., $\gfr = \hfr' \oplus \hfr'$ and $\hfr = \text{diag}(\hfr')$ for some simple Lie algebra $\hfr'$.
If $\Ocal_{\text{min}} \subset \hfr'$ is the minimal nilpotent orbit (see Remark \ref{rmk: nilpotent orbit in s}), then $O_{\mfr} = \PP (\{(v \oplus (-v)) \in \gfr : v \in \Ocal_{\text{min}}\})$ and $Z_{\mfr} = \PP(\Ocal_{\text{min}} \oplus \Ocal_{\text{min}})$.
In particular, $O_{\mfr}$ is a Legendrian subvariety of $Z_{\mfr}$, and its dimension is given in Table~\ref{table: symmetric isotropy irreducible pairs of different rank with g simple}.
\end{proposition}
\begin{proof}
Let $H'$ be the adjoint group of $\hfr'$ so that $G_{\text{ad}} = H' \times H'$, and then $O_{\mfr}$ and $Z_{\mfr}$ are homogeneous under the action of $\text{diag}(H')$ and $H' \times H'$, respectively.
Let $E_{\delta}$ be the highest root vector of $\hfr'(\simeq \hfr)$.
    Since $\mfr = \{x \oplus (-x) \in \gfr : x \in \hfr'\}$, $O_{\mfr}$ and $Z_{\mfr}$ contain $[E_{\delta} \oplus (-E_{\delta})]$, and hence $O_{\mfr} = \PP (\{(v \oplus (-v)) \in \gfr : v \in \Ocal_{\text{min}}\})$ and $Z_{\mfr} = \PP(\Ocal_{\text{min}} \oplus \Ocal_{\text{min}})$.
    Thus we have $\dim O_{\mfr} = \dim \Ocal_{\text{min}} -1$ and $\dim Z_{\mfr} = 2 \dim \Ocal_{\text{min}} - 1$.
\end{proof}

It remains to consider the case where $\gfr$ is simple.
In such a case, the following observation is useful:

\begin{proposition} \label{prop: rho is a root sp of g if it is not a root of h}
    If $\rho$ is not a root of $\hfr$ (see Corollary \ref{coro: corollary of table, characterization of B3 G2}), then $\mfr_{\rho}$ is a root space of $\gfr$ with respect to $\tfr$.
\end{proposition}

\begin{proof}
    Observe that since the $\rho$-weight space $\gfr_{\rho}$ of $\gfr$ (as a $\tfr_{H}$-representation) is $\mfr_{\rho} \oplus \hfr_{\rho}$, if $\rho$ is not a root of $\hfr$, then $\mfr_{\rho} = \gfr_{\rho}$.
    Since $\tfr_{H} \le \tfr$, a weight space of $\gfr$ as a $\tfr_{H}$-representation is generated by root spaces of $\gfr$.
    Since $\mfr_{\rho}$ is a highest weight space, it is of dimension 1, and hence it coincides with a root space.
\end{proof}

The following proposition considers the remaining cases of symmetric subalgebras in Theorems \ref{main thm: adjoint}--\ref{main thm: semi simple}, not covered in Propositions \ref{prop: Legendrian for IHSS}--\ref{prop:Om when g is not simple}:

\begin{proposition} \label{prop: highest weight orbits for symm iso irre pair}
    If $(\gfr,\,\hfr)$ is symmetric and $\gfr$ is simple, then $O_{\mfr}$ is a Legendrian subvariety of $Z_{\mfr}$.
    A list of $Z_{\mfr}$ for such $(\gfr,\,\hfr)$ is given in Table \ref{table: isotropy irreducible pairs of equal rank} (when $\rank(\hfr) = \rank(\gfr)$) and Table \ref{table: symmetric isotropy irreducible pairs of different rank with g simple} (when $\rank(\hfr) < \rank(\gfr)$).
\end{proposition}

\begin{proof}
    We use the well-known classification of symmetric varieties, which can be found in \cite[\S8.10--8.11]{Wolf1984SpaceConstant}, and summarized in Tables \ref{table: isotropy irreducible pairs of equal rank}--\ref{table: symmetric isotropy irreducible pairs of different rank with g simple}.

    If $\rank(\hfr) = \rank(\gfr)$, then $(\gfr,\,\hfr)$ belongs to Table \ref{table: isotropy irreducible pairs of equal rank} (up to conjugacy).
    In this case, $\rho$ is given in the second column, and hence the marked Dynkin diagram of $O_{\mfr}$ (the third column) and its dimension (the fourth column) follow.
    Moreover, since $Z = Z_{\text{long}}$ ($Z_{\text{short}}$, respectively) if and only if $\rho$ is long (short, respectively), the last column of Table \ref{table: isotropy irreducible pairs of equal rank} follows.
    By comparing $\dim O_{\mfr}$ and $\dim Z_{\mfr}$, we conclude that $O_{\mfr}$ is always Legendrian.

    Next, assume that $\rank(\hfr) < \rank(\gfr)$ so that $(\gfr,\,\hfr)$ belongs to Table \ref{table: symmetric isotropy irreducible pairs of different rank with g simple}.
    Again, $\rho$ is given in the second column, and hence $\dim O_{\mfr}$ follows, as listed in the third column.
    If $(\gfr,\,\hfr)$ is not one of ($A_{2l-1}$, $C_{l}$) ($l \ge 2$), ($D_{p+1}$, $B_{p}$) ($p \ge 2$) and ($E_{6}$, $F_{4}$), then by Table \ref{table: symmetric isotropy irreducible pairs of different rank with g simple}, $\rho$ is not a root of $\hfr$ and $\gfr$ is of type $ADE$, and hence $O_{\mfr} \subset Z_{\text{long}}$ by Proposition \ref{prop: rho is a root sp of g if it is not a root of h}.
    Again by comparing the dimensions, we conclude that $O_{\mfr}$ is a Legendrian subvariety of $Z_{\text{long}}$.

    Now it remains to consider ($A_{2l-1}$, $C_{l}$) ($l \ge 2$), ($D_{p+1}$, $B_{p}$) ($p \ge 2$) and ($E_{6}$, $F_{4}$).
    To complete the proof, we use their constructions in terms of diagram folding, see \cite[Example 2 and Theorem 5.15, \S X.5]{Helgason1979DifferentialGeometry}.
    Consider a diagram automorphism of order 2 on the Dynkin diagram of $\gfr$, given by switching nodes as follows:
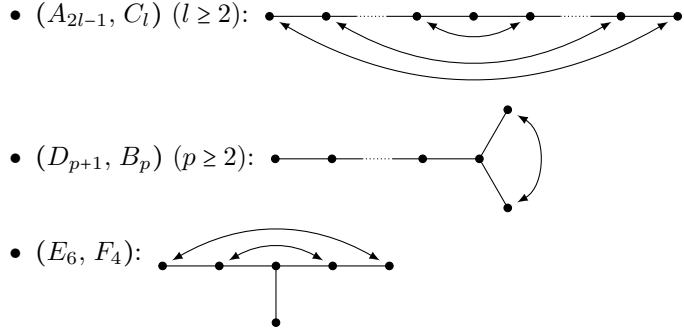
\begin{itemize}
    \item $(A_{2l-1},\, C_{l})$ ($l \ge 2$): \dynkin[edge length=.75cm, involutions={17;26;35}]{A}{**.***.**}
    
    \item $(D_{p+1},\, B_{p})$ ($p \ge 2$): \dynkin[edge length=.75cm, involutions={[in=120,out=60, relative]56}]{D}{**.****}
    
    \item $(E_{6},\, F_{4})$: \dynkin[edge length=.75cm, involutions={16;35}, upside down]E6
\end{itemize}
By identifying the nodes connected by arrows so that each identified node represents a short simple root, we obtain the Dynkin diagram of $\hfr$.
Furthermore, it induces an outer involution of $\gfr$ such that the fixed-point-locus is $\hfr$, and $\tfr$ is stable under the involution.

To be precise, let us denote simple roots of $\hfr$ and $\gfr$ by $\alpha_{i}$ and $\beta_{i}$ (labeled as in Section \ref{section: Tables}).
If the nodes corresponding to $\beta_{i}$ and $\beta_{j}$ are connected by an arrow and folded to a node corresponding to $\alpha_{k}$, then $\beta_{i} |_{\tfr_{H}} = \beta_{j}|_{\tfr_{H}} = \alpha_{k}$.
Here is a list of such triples:
\begin{itemize}
    \item $(A_{2l-1},\, C_{l})$ ($l \ge 2$): $\beta_{i}|_{\tfr_{H}} = \beta_{2l-i}|_{\tfr_{H}} = \alpha_{i}$, $1 \le i \le l-1$.
    
    \item $(D_{p+1},\, B_{p})$ ($p \ge 2$): $\beta_{p}|_{\tfr_{H}} = \beta_{p+1}|_{\tfr_{H}} = \alpha_{p}$.
    
    \item $(E_{6},\, F_{4})$: $\beta_{i} |_{\tfr_{H}} = \beta_{6-i}|_{\tfr_{H}} = \alpha_{i}$, $i=1,\,2$.
\end{itemize}

Next, consider the orthogonal decomposition $\gfr = \hfr \oplus \mfr$, which gives $\gfr_{\rho} = \hfr_{\rho} \oplus \mfr_{\rho}$, orthogonal decomposition of the $\rho$-weight space $\gfr_{\rho}$ (as a $\tfr_{H}$-representation).
In those exceptions, $\rho$ is always the dominant short root (Corollary \ref{coro: corollary of table, characterization of B3 G2}), and hence $\gfr_{\rho}$ is of dimension 2.
It means that $\gfr_{\rho}$ is generated by two root spaces, associated to two roots $\gamma_{1}$ and $\gamma_{2}$ of $\gfr$ such that $\gamma_{i}|_{\tfr_{H}} = \rho$.
Thus $v_{\rho} = a_{1} \cdot E_{\gamma_{1}} + a_{2} \cdot E_{\gamma_{2}}$ for some $a_{i} \in \CC$.
In fact, both $a_{1}$ and $a_{2}$ are nonzero, since $2 \dim O_{\mfr} + 1 > \dim Z_{\text{long}}$ (cf. Table~\ref{table: symmetric isotropy irreducible pairs of different rank with g simple}), and hence $Z_{\mfr} \not=Z_{\text{long}}$.
Now we consider case by case.
\begin{itemize}
    \item $(A_{2l-1},\, C_{l})$ ($l \ge 2$): In this case, $\gamma_{1} := \beta_{1} + \cdots + \beta_{2l-2}$ and $\gamma_{2} := \beta_{2} + \cdots + \beta_{2l-1}$.
        Let us identify $\gfr = \mathfrak{sl}(2l)$ with the algebra of traceless matrices.
        We may choose $\tfr$ as the subalgebra of the diagonal matrices, and then $\beta_{i} = \epsilon_{i} - \epsilon_{i+1}$ where $\epsilon_{i} : \tfr \rightarrow \CC$ is the linear functional that assigns the $i$th entry.
        The roots of $\gfr$ are given by $\epsilon_{i} - \epsilon_{j}$ ($1 \le i \not= j \le 2l$), and their root spaces are generated by $e_{ij}$, the elementary matrix with a unique nonzero entry at the $i$th row and the $j$th column.
        Thus $v_{\rho} = a_{1} \cdot E_{\gamma_{1}} + a_{2} \cdot E_{\gamma_{2}} = a_{1} e_{1,\,2l-1} + a_{2} e_{2,\,2l}$, and it is easy to show that it is conjugate to a Jordan matrix
        \[
            \begin{pmatrix}
                J_{2} & & & & \\
                & J_{2} & & & \\
                & & J_{1} & & \\
                & & & \ddots & \\
                & & & & J_{1}
            \end{pmatrix}.
        \]
        Thus $[\mfr_{\rho}] \in Z_{[2, \, 2,\, 1 ,\, \cdots,\, 1]} = Z_{[2^{2}, \, 1^{2l-4}]}$.
    
    \item $(D_{p+1},\, B_{p})$ ($p \ge 2$): In this case, $\gamma_{1} := \beta_{1} + \cdots + \beta_{p-1} + \beta_{p}$ and $\gamma_{2} := \beta_{1} + \cdots + \beta_{p-1} + \beta_{p+1}$.
        While we can proceed as in the previous case, instead, let us introduce more elementary argument.

        Let $(\gfr,\,\hfr) = (\mathfrak{so}(n+1), \, \mathfrak{so}(n))$, $n \ge 2$.
        We may consider $\gfr$ as the algebra of skew-symmetric $(n+1) \times (n+1)$ matrices, and $\hfr$ as the subalgebra of matrices whose $(n+1)$-th row and $(n+1)$-th column are zero.
        Since the Killing form of $\gfr$ is given as the trace form, the orthogonal complement $\mfr$ of $\hfr$ is consisting of matrices of form
        \[
            \begin{pmatrix}
                 & & & & x_{1} \\
                & & & & x_{2} \\
                & & & & \vdots  \\
                & & & & x_{n} \\
                -x_{1}& -x_{2} & \cdots & -x_{n} & 0
            \end{pmatrix}.
        \]
        Moreover, as an $\mathfrak{so}(n)$-representation, it is isomorphic to the standard one.
        Thus the highest weight orbit $O_{\mfr}$, which is the smooth quadric defined by $\sum_{i} x_{i}^{2} = 0$, contains an element
        \[
            \begin{pmatrix}
                 & & & & 1 \\
                & & & & \sqrt{-1} \\
                & & & & \vdots  \\
                & & & & 0 \\
                -1& -\sqrt{-1} & \cdots & 0 & 0
            \end{pmatrix}
        \]
        whose Jordan normal form is
        \[
            \begin{pmatrix}
                J_{3} & & & \\
                & J_{1} & & \\
                & & \ddots & \\
                & & & J_{1}
            \end{pmatrix}.
        \]
        Thus $O_{\mfr} \subset Z_{[3,\,1^{n-2}]}$ (which is equal to $Z_{\text{short}}$ when $n$ is even).
    
    \item $(E_{6},\, F_{4})$: In this case, $\gamma_{1} := \beta_{1} + \beta_{2} + 2 \beta_{3} + 2 \beta_{4} + \beta_{5} + \beta_{6}$ and $\gamma_{2} := \beta_{1} + 2\beta_{2} + 2 \beta_{3} + \beta_{4} + \beta_{5} + \beta_{6}$.
    For another root $\gamma_{0} := \beta_{1} + \beta_{2} + 2 \beta_{3} + \beta_{4} + \beta_{5} + \beta_{6}$ and the reflection $s_{\gamma_{0}}$ with respect to the hyperplane defined by $\gamma_{0}$, we have
    \[
        s_{\gamma_{0}}(\beta_{2}) = \beta_{2} + \gamma_{0} = \gamma_{2}, \quad s_{\gamma_{0}}(\beta_{4}) = \beta_{4} + \gamma_{0} = \gamma_{1}.
    \]
    This shows that $O_{\mfr}$ is contained in the nilpotent orbit containing $[E_{\beta_{4}} + E_{\beta_{2}}]$, i.e., $Z_{2A_{1}}$.
\end{itemize}
Again, by comparing $\dim O_{\mfr}$ and $\dim Z_{\mfr}$, we conclude that $O_{\mfr}$ is always a Legendrian subvariety of $Z_{\mfr}$.
\end{proof}

Finally, we consider the case where $(\gfr,\,\hfr)$ is not symmetric (corresponding to Theorems \ref{main thm: adjoint}(2) and \ref{main thm: semi simple}(2.g)).
Recall that if $(\gfr,\,\hfr)$ is not symmetric, then $\gfr$ is simple, $\hfr$ is semi-simple and $\rank(\hfr) < \rank(\gfr)$ (Theorem \ref{thm: classification of isotropy irre pair}).
In this case, $W$, the set of nonzero $\tfr_{H}$-weights of $\mfr$, is contained in the root lattice $\ZZ \cdot R_{\hfr}$ of $\hfr$.
Indeed, since $\tfr_{H} \oplus \mfr_{0}$ is the centralizer of $\tfr_{H}$ in $\gfr$, we have $\mfr_{0} \not= 0$.
By the irreducibility of $\mfr$, the $\hfr$-representation generated by $\mfr_{0}$ must be equal to $\mfr$, so $W \subset \ZZ \cdot R_{\hfr}$.
In particular, $W \subset \QQ \cdot R_{\hfr}$.

\begin{lemma} \label{lem: borel subalgebra of g}
For each $w \in \QQ \cdot R_{\hfr}$, let $s(w) \in \ZZ$ be the sum of the coefficients in its expression with respect to the simple roots of $\hfr$.
If $W \subset \QQ \cdot R_{\hfr}$, then we have the following:
    \begin{enumerate}
        \item $\tfr_{H} \oplus \mfr_{0} \oplus \bigoplus_{w \in W \, : \, s(w) = 0} \mfr_{w}$ is a reductive subalgebra of $\gfr$.
        \item The vector subspace spanned by $\bfr_{H}$, $\bigoplus_{w \in W \, : \, s(w) > 0} \mfr_{w}$ and a Borel subalgebra of $\tfr_{H} \oplus \mfr_{0} \oplus \bigoplus_{w \in W \, : \, s(w) = 0} \mfr_{w}$ containing $\tfr$ is a Borel subalgebra of $\gfr$.
    \end{enumerate}
\end{lemma}
\begin{proof}
    \begin{enumerate}
        \item It is clear that $\kfr_{0}:= \tfr_{H} \oplus \mfr_{0} \oplus \bigoplus_{w \in W \, : \, s(w) = 0} \mfr_{w}$ is a subalgebra of $\gfr$.
        For algebraicity, observe that the subspace $\bigoplus_{w \in W \, : \, s(w) = 0} \mfr_{w}$ is contained in the derived subalgebra $[\kfr_{0},\, \kfr_{0}]$.
        Thus $\kfr_{0}$ is generated by the algebraic subalgebras $[\kfr_{0},\, \kfr_{0}]$ and $\tfr_{H} \oplus \mfr_{0}$, and hence $\kfr_{0}$ is also algebraic.
        Furthermore, $\kfr_{0}$ is reductive since the restriction $b_{\gfr}|_{\kfr_{0}}$ is non-degenerate and by \cite[Theorem 2, \S 4.1.1]{OnishchikVinberg1990LieGroups}.
        \item Let $\bfr_{\kfr_{0}}$ be a Borel subalgebra of $\kfr_{0}$ containing $\tfr$, and put
        \[
            \bfr := (\bfr_{H} + \bfr_{\kfr_{0}}) \oplus \bigoplus_{w \in W \, : \, s(w) > 0} \mfr_{w} = \ufr_{H} \oplus \bfr_{\kfr_{0}} \oplus \bigoplus_{w \in W \, : \, s(w) > 0} \mfr_{w}
        \]
        where $\ufr_{H}$ is the unipotent radical of $\bfr_{H}$.
        One can easily show that $\bfr$ is a subalgebra of $\gfr$.
        Since $\ufr_{H} \oplus \bigoplus_{w \in W \, : \, s(w) > 0} \mfr_{w}$ is a solvable ideal of $\bfr$, we see that $\bfr$ is solvable.
        
        To see the maximality of $\bfr$, let $\tilde{\bfr}$ be a solvable subalgebra of $\gfr$ containing $\bfr$ properly.
        Since $\bfr$ contains $\tfr$, both $\bfr$ and $\tilde{\bfr}$ are spanned by $\tfr$ and root vectors of $\gfr$.
        Thus there is a root $\beta \in R_{\gfr}$ such that $\gfr_{\beta} \setminus \{0\} \subset \tilde{\bfr} \setminus \bfr$.
        By its definition, $s(\beta|_{\tfr_{H}}) \le 0$.
        \begin{itemize}
            \item If $s(\beta|_{\tfr_{H}}) = 0$, then $\gfr_{\beta} \subset \kfr_{0}$. Since $\bfr_{\kfr_{0}}$ is a Borel subalgebra of $\kfr_{0}$, we have $\gfr_{\beta} \subset \tilde{\bfr} \cap \kfr_{0} = \bfr_{\kfr_{0}} \le \bfr$, a contradiction.
            \item If $s(\beta|_{\tfr_{H}}) < 0$, then $\gfr_{-\beta} \subset \bfr$, and hence the $\mathfrak{sl}(2)$-subalgebra $\gfr_{\beta} \oplus [\gfr_{\beta}, \, \gfr_{-\beta}] \oplus \gfr_{-\beta}$ is contained in $\tilde{\bfr}$, a contradiction.
        \end{itemize}
        Therefore $\bfr$ is a Borel subalgebra.
    \end{enumerate}
\end{proof}

Before we proceed further, let us record another corollary of Table \ref{table: classification of g h}.

\begin{coro} \label{coro: corollary of table, information on s(w)}
    Assume that $(\gfr,\,\hfr)$ is not symmetric.
    In the notation of Lemma \ref{lem: borel subalgebra of g}, for the highest root $\delta^{\hfr_{1}}$ of a simple factor $\hfr_{1}$ of $\hfr$, we have
        \begin{itemize}
            \item $s(\rho) < s (\delta^{\hfr_{1}})$ if $(\gfr,\, \hfr, \hfr_{1})$ is one of $(B_{3},\, G_{2}, \, G_{2})$, $(E_{7}, \, A_{1} \oplus F_{4}, \, F_{4})$,
            \item $s(\rho) = s (\delta^{\hfr_{1}})$ if $(\gfr,\, \hfr, \hfr_{1})$ is one of $(D_{2n}, \, A_{1} \oplus C_{n}, \, C_{n})$ ($n\ge 3$), $(F_{4}, \, A_{1} \oplus G_{2}, \, G_{2})$, $(E_{6}, \, A_{2} \oplus G_{2}, \, G_{2})$, $(E_{8}, \, G_{2} \oplus F_{4}, \, F_{4})$, and
            \item $s(\rho) > s(\delta^{\hfr_{1}})$ otherwise.
        \end{itemize}
\end{coro}

The following proposition considers all the remaining cases, i.e., Theorems \ref{main thm: adjoint}(2) and \ref{main thm: semi simple}(2.g):

\begin{proposition} \label{prop: rho is a highest root sp}
    If $(\gfr,\, \hfr)$ is not symmetric, then $Z_{\mfr} = Z_{\text{long}}$ for $(\gfr,\,\hfr) \not= (B_{3},\, G_{2})$, and $Z_{\mfr} = Z_{[3,\,2^{2}]}$ for $(\gfr,\,\hfr) = (B_{3},\, G_{2})$.
    Moreover, $O_{\mfr}$ is a Legendrian subvariety of $Z_{\mfr}$ if and only if the last column in Table \ref{table: classification of g h} is marked as \lq Yes\rq.
\end{proposition}

\begin{proof}
    As noted before, $\dim O_{\mfr}$ follows from the highest weight $\rho$ of $\mfr$ given in Table \ref{table: classification of g h}.
    Thus the second statement follows from the first statement, and so it is enough to find $Z_{\mfr}$ in each case.

    First, assume that $(\gfr,\, \hfr) \not= (B_{3}, \, G_{2})$.
    By Proposition \ref{prop: rho is a root sp of g if it is not a root of h} and Corollary \ref{coro: corollary of table, characterization of B3 G2}, $\mfr_{\rho}$ is a root space of $\gfr$ with respect to $\tfr$.
    In particular, if the Dynkin diagram of $\gfr$ is simply laced, then $\mfr_{\rho}$ is a long root space (with respect to $\tfr$).
    Thus $Z_{\mfr} = Z_{\text{long}}$.

    For the remaining cases ($\not= (B_{3}, \, G_{2})$), let $\bfr$ be a Borel subalgebra of $\gfr$ constructed in Lemma \ref{lem: borel subalgebra of g}.
    Note that for every $w \in W \cup \{0\}$ different from $\rho$, we have $s(w) < s(\rho)$.
    Moreover, by Corollary \ref{coro: corollary of table, information on s(w)}, if
    \begin{itemize}
        \item the Dynkin diagram of $\gfr$ is not simply laced, and
        \item $(\gfr, \, \hfr) \not= (F_{4}, \, A_{1} \oplus G_{2})$,
    \end{itemize}
    then $s(\rho) > s(\alpha)$ for all $\alpha \in R_{\hfr}$, and hence $\mfr_{\rho}$ is $\bfr$-stable.
    That is, $\mfr_{\rho}$ is the highest root space of $\gfr$ with respect to $\bfr$.
    Therefore $\mfr_{\rho}$ is a long root space and thus $Z_{\mfr} = Z_{\text{long}}$, if $(\gfr,\,\hfr) \not= (F_{4},\, A_{1} \oplus G_{2})$.

    Now consider the case $(\gfr,\, \hfr) = (F_{4}, \, A_{1} \oplus G_{2})$.
    In this case, $s(\rho) = s(\delta) (= 5) > s(w)$ for the highest root $\delta$ of the $G_{2}$ factor and any $w \in (R_{\hfr} \setminus \{\delta\}) \cup (W \setminus\{\rho\})$.
    Moreover, since $\mfr$ is isomorphic to the tensor product of an irreducible $A_{1}$-representation and the first fundamental $G_{2}$-representation (both of which have 1-dimensional weight spaces), each weight space $\mfr_{w}$ is of dimension 1.
    Thus if $w \in W \setminus R_{\hfr}$, then $\mfr_{w}$ is a root space of $\gfr$.
    Now for $\kfr_{0} := \tfr_{H} \oplus \mfr_{0} \oplus \bigoplus_{w \in W \, : \, s(w) = 0} \mfr_{w}$ (as in Lemma \ref{lem: borel subalgebra of g}),
    \[
        \left[ \kfr_{0} , \, \hfr_{\delta} \oplus \mfr_{\rho} \right] = \hfr_{\delta} \oplus \mfr_{\rho}.
    \]
    Let $\bfr_{\kfr_{0}}$ be a Borel subalgebra of $\kfr_{0}$ containing $\tfr$.
    If $[\bfr_{\kfr_{0}}, \, \mfr_{\rho}] \subset \mfr_{\rho}$, then $\mfr_{\rho}$ is stable under the Borel subalgebra
    \[
        (\bfr_{H} + \bfr_{\kfr_{0}}) \oplus \bigoplus_{w \in W \, : \, s(w) > 0} \mfr_{w}
    \]
    of $\gfr$ (Lemma \ref{lem: borel subalgebra of g}). 
    If $[\bfr_{\kfr_{0}}, \, \mfr_{\rho}] \not\subset \mfr_{\rho}$, then $\bfr_{\kfr_{0}}$ contains the $(\delta - \rho)$-weight space $\mfr_{\delta - \rho}$, which is a root space of $\gfr$ since $(\delta-\rho) \not\in R_{\hfr}$.
    Since the opposite Borel subalgebra $\bfr_{\kfr_{0}}^{-} \le \kfr_{0}$ does not contain $\mfr_{\delta - \rho}$, $\mfr_{\rho}$ is stable under the Borel subalgebra
    \[
        (\bfr_{H} + \bfr_{\kfr_{0}}^{-}) \oplus \bigoplus_{w \in W \, : \, s(w) > 0} \mfr_{w}
    \]
    of $\gfr$ (Lemma \ref{lem: borel subalgebra of g}).
    Therefore in any cases, $\mfr_{\rho}$ is a long root space, and thus $Z_{\mfr} = Z_{\text{long}}$.

    Finally, assume that $(\gfr,\,\hfr)=(B_{3},\, G_{2})$.
The embedding $\hfr \hookrightarrow \gfr$ can be constructed as follows:
Put $\tilde{\gfr} := D_{4}$ and denote its simple roots by $\tilde{\beta}_{1}$, ..., $\tilde{\beta}_{4}$.
If $\sigma_{2}$ and $\sigma_{3}$ are the automorphisms of $\tilde{\gfr}$ induced by diagram automorphisms
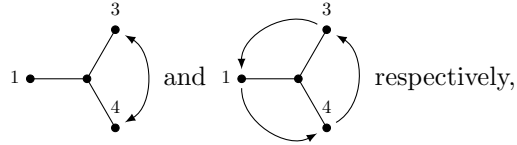
\begin{center}
    \begin{dynkinDiagram}[edge length=.75cm, labels={1,,3,4}, label directions={left,,above,above}, involutions={[in=120,out=60, relative]34}]D{****}
    \end{dynkinDiagram}
    and
    \begin{dynkinDiagram}[edge length=.75cm, labels={1,,3,4},label directions={left,,above,above}]D{****}
        \draw[-latex, in=-120,out=-60, relative] (root 3) to (root 1);
        \draw[-latex, in=-120,out=-60, relative] (root 1) to (root 4);
        \draw[-latex,in=-120,out=-60, relative] (root 4) to (root 3);
    \end{dynkinDiagram}
    respectively,
\end{center}
then $\gfr$ and $\hfr$ are fixed-point-loci of $\sigma_{2}$ and $\sigma_{3}$ in $\tilde{\gfr}$, respectively.
Thus we can choose $\tfr$ and $\tfr_{H}$ so that for the simple roots indexed as in the diagrams
    \begin{center}
        ${\dynkin[labels={\alpha_{1},\alpha_{2}}] G2} \quad \text{and} \quad {\dynkin[labels={\beta_{1},\beta_{2},\beta_{3}}] B3}$,
    \end{center}
we have
\[
    \tilde{\beta}_{1}|_{\tfr} = \beta_{1}, \quad \tilde{\beta}_{2}|_{\tfr} = \beta_{2}, \quad \tilde{\beta}_{3}|_{\tfr} = \tilde{\beta}_{4}|_{\tfr} = \beta_{3},
\]
and
\[
    \tilde{\beta}_{2}|_{\tfr_{H}} = \alpha_{2}, \quad \tilde{\beta}_{1}|_{\tfr_{H}}= \tilde{\beta}_{3}|_{\tfr_{H}} = \tilde{\beta}_{4}|_{\tfr_{H}} = \alpha_{1}.
\]
Since $\rho = 2\alpha_{1}+\alpha_{2}$, there are exactly two roots of $\gfr$ whose restrictions on $\tfr_{H}$ are equal to $\rho$, namely $\beta_{1} + \beta_{2} + \beta_{3}$ and $\beta_{2} + 2 \beta_{3}$.
It means that $\mfr_{\rho}$ is generated by $a_{1} E_{\beta_{1} + \beta_{2} + \beta_{3}} + a_{2} E_{\beta_{2} + 2 \beta_{3}}$ for some $a_{i} \in \CC$.
In fact, by \cite[Remark 5.4.2, Theorem 5.1.2 and Corollary 6.1.4]{CollingwoodMcGovern1993NilpotentOrbits}, there are 6 nilpotent orbits in $\PP(\gfr)$
\[
    Z_{[7]}, \quad Z_{[5,\,1^{2}]}, \quad Z_{[3^{2}, \, 1]},\quad Z_{[3,\, 2^{2}]}, \quad Z_{[3,\, 1^{4}]} (=Z_{\text{short}}),\quad Z_{[2^{2},\, 1^{3}]} (=Z_{\text{long}})
\]
of dimension
\[
    17, \quad 15,\quad 13, \quad 11, \quad 9, \quad 7,
\]
respectively.
Since $O_{\mfr} (\simeq \QQ^{5})$ is of dimension 5, $Z_{\mfr} \not= Z_{\text{short}},\, Z_{\text{long}}$, and hence $a_{1},\, a_{2}\not=0$.
Moreover, since the maximal torus of $G$ acts on $\PP(E_{\beta_{1} + \beta_{2} + \beta_{3}},\, E_{\beta_{2} + 2\beta_{3}}) \setminus \{[E_{\beta_{1} + \beta_{2} + \beta_{3}}],\, [E_{\beta_{2} + 2\beta_{3}}]\}$ transitively, we see that
\[
    \PP(E_{\beta_{1} + \beta_{2} + \beta_{3}},\, E_{\beta_{2} + 2\beta_{3}}) \setminus \{[E_{\beta_{1} + \beta_{2} + \beta_{3}}],\, [E_{\beta_{2} + 2\beta_{3}}]\} \subset Z_{\mfr}.
\]

Now recall that $\gfr= \mathfrak{so}(7)$ can be identified with the algebra of skew-symmetric $7 \times 7$ matrices.
Following \cite[\S III.8]{Helgason1979DifferentialGeometry}, we consider $\tfr$ generated by $h_{i} := e_{2i-1,\, 2i} - e_{2i,\, 2i-1}$ ($1 \le i \le 3$) where $e_{ij}$ denotes the elementary matrix with 1 at the intersection of the $i$th row and the $j$th column.
The roots are given by $\pm \epsilon_{i}$ ($1\le i \le 3$), $\epsilon_{i} - \epsilon_{j}$ ($1 \le i\not= j \le 3$) and $\pm (\epsilon_{i} + \epsilon_{j})$ ($1 \le i < j \le 3$) where $\epsilon_{i}$ is defined by setting $\epsilon_{i}(h_{j}) = -\sqrt{-1}\delta_{ij}$.
For $\beta_{1} := \epsilon_{1} - \epsilon_{2}$, $\beta_{2} := \epsilon_{2} - \epsilon_{3}$ and $\beta_{3} := \epsilon_{3}$, the root spaces associated to $\beta_{1} + \beta_{2} + \beta_{3} (=\epsilon_{1})$ and $\beta_{2} + 2 \beta_{3} (=\epsilon_{2} + \epsilon_{3})$ are generated by matrices
\[
    \begin{pmatrix}
        &&&&&& 1 \\
        &&&&&& -\sqrt{-1} \\
        &&&&&& \\
        &&&&&& \\
        &&&&&& \\
        &&&&&& \\
       -1 & \sqrt{-1} &&&&& \\
    \end{pmatrix} \quad \text{and} \quad \begin{pmatrix}
        &&&&&& \\
        &&&&&& \\
        &&&&1&-\sqrt{-1}& \\
        &&&&-\sqrt{-1}&-1& \\
        &&-1&\sqrt{-1}&&& \\
        &&\sqrt{-1}&1&&& \\
        &&&&&& \\
    \end{pmatrix},
\]
respectively.
Therefore $Z_{\mfr}$ contains a point represented by a matrix
\[
    \begin{pmatrix}
        &&&&&& 1 \\
        &&&&&& -\sqrt{-1} \\
        &&&&1&-\sqrt{-1}& \\
        &&&&-\sqrt{-1}&-1& \\
        &&-1&\sqrt{-1}&&& \\
        &&\sqrt{-1}&1&&& \\
        -1& \sqrt{-1} &&&&& \\
    \end{pmatrix}.
\]
Its row rank is 4, its square is of row rank 1, and its 3rd power is zero, and so $Z_{\mfr} = Z_{[3,\,2^{2}]}$.
\end{proof}

\begin{exa} \label{example: legendrian Om}
Here are some examples of non-symmetric $(\gfr,\,\hfr)$ with $O_{\mfr} \subset Z_{\text{long}}$ Legendrian.
\begin{enumerate}
    \item For non-symmetric $(\gfr,\,\hfr)$ with $\gfr = \mathfrak{sp}(2r)$, $r \ge 1$, $O_{\mfr}$ is always a Legendrian subvariety of $Z_{\text{long}} \simeq \PP^{2r-1}$.
    As noted in the introduction, it is well known that every equivariant Legendrian embeddings of rational homogeneous spaces into $\PP^{2r-1}$ can be obtained as a subadjoint variety (see \cite{Buczynski2008AlgebraicLegendrian, LandsbergManivel2007LegendrianVarieties}).
    In fact, $O_{\mfr}$'s for No. 6, 7, 8, 9, 10, and $11_{n}$ in Table \ref{table: classification of g h} are subadjoint varieties defined by $G_{2}$, $F_{4}$, $E_{6}$, $E_{7}$, $E_{8}$ and $\mathfrak{so}(n+4)$, respectively.
    This shows that every non-linear subadjoint variety arises from non-symmetric $(\gfr,\,\hfr)$ with $\gfr$ of type $C$.
    (Remark that the linear subspace $\PP^{n} \subset \PP^{2n+1}$, the subadjoint variety defined by $A_{n+2}$, can be obtained by Proposition \ref{prop: Legendrian for IHSS}, by taking $S_{\text{ad}}/P$ as $\text{LG}(n+1,\, \CC^{2n+2})$.)

    \item There are three infinite non-symmetric families of $(\gfr,\,\hfr)$ with $O_{\mfr} \subset Z_{\text{long}}$ Legendrian: $(A_{2p-1}, \, A_{p-1} \oplus A_{1})$ ($p \ge 3$), $(C_{n}, \, A_{1} \oplus \mathfrak{so}(n))$ ($n \ge 3$), and $(D_{2n}, \, A_{1} \oplus C_{n})$ ($n \ge 3$).
    As in Example \ref{example: non-symm isotropy irre embedding}, these are arising from the isotropy representations of $\text{Gr}(2,\, \CC^{p+2})$, the adjoint variety of $\mathfrak{so}(n+4)$, and $\text{IG}(2,\, \CC^{2n+4})$, respectively.
    On the other hand, these three families are all possible non-symmetric $(\gfr,\, \hfr)$ with $\gfr$ classical and $\hfr$ not simple as proved in \cite{Wolf1984CorrectionGeometry}.
\end{enumerate}
\end{exa}

\subsection{Homogeneous Legendrian subvarieties of adjoint varieties as linear sections}

While $O_{\mfr} \subset Z_{\mfr}$ is Legendrian for all symmetric $(\gfr,\,\hfr)$ (see Remark \ref{rmk: Legendrian when symm}), as indicated in Table \ref{table: classification of g h}, $O_{\mfr} \subset Z_{\mfr}$ is not always Legendrian if $(\gfr,\,\hfr)$ is not symmetric.
In this section, we give a geometric characterization of the pairs $(\gfr,\,\hfr)$ such that $O_{\mfr}$ is a Legendrian subvariety of the adjoint variety.
The results in this section are not used in the proof of our main Theorems \ref{main thm: adjoint}--\ref{main thm: semi simple}.

First, as we have seen in Propositions \ref{prop: highest weight orbits for symm iso irre pair} and \ref{prop: rho is a highest root sp}, when $\gfr$ is simple, $Z_{\mfr}$ is the adjoint variety $Z_{\text{long}}$ for $\gfr$ with only few exceptions.
More precisely:
\begin{theorem} \label{thm: classification of Om and Zm of iso irre var}
    If $\gfr$ is simple, then $Z_{\mfr}$ is the adjoint variety $Z_{\text{long}} \subset \PP(\gfr)$ unless $(\gfr ,\,\hfr)$ is one of ($A_{2l-1}$, $C_{l}$) ($l \ge 2$), ($C_{l}$, $C_{p} \oplus C_{l-p}$) ($1 \le p \le l-1$), ($\mathfrak{so}(l)$, $\mathfrak{so}(l-1)$) ($l \ge 5$), ($F_{4}$, $B_{4}$), ($E_{6}$, $F_{4}$), and ($B_{3}$, $G_{2}$).
\end{theorem}

\begin{rmk}
    It would be interesting to observe that the exceptions in Theorem \ref{thm: classification of Om and Zm of iso irre var} appear in the classification of \emph{shared orbit pairs} due to Brylinski and Kostant \cite{BrylinskiKostant1994NilpotentOrbits} (see also \cite[Example 2.7.a]{FuJuteauLevySommers2023LocalGeometry}).
    Here, a \emph{shared orbit pair} means a pair $(\Ocal_{1},\, \Ocal_{2})$ of nilpotent orbits $\Ocal_{i} \subset \gfr_{i}$ ($i=1,\,2$) for reductive Lie algebras $\gfr_{1} < \gfr_{2}$ such that there is a $G_{1}$-equivariant finite morphism $\overline{\Ocal_{2}} \rightarrow \overline{\Ocal_{1}}$.
    In fact, by combining the classifications, the exceptions in Theorem \ref{thm: classification of Om and Zm of iso irre var} can be characterized as isotropy irreducible pairs $(\gfr,\,\hfr)$ with $\gfr$ simple and having a shared orbit pair.
    A geometric explanation for this coincidence is not known to the author.
    Nonetheless, we shall use other shared orbit pairs related to ($C_{l}$, $C_{p} \oplus C_{l-p}$) ($1 \le p \le l-1$), ($\mathfrak{so}(l)$, $\mathfrak{so}(l-1)$) ($l \ge 5$), ($F_{4}$, $B_{4}$) and ($B_{3}$, $G_{2}$) to construct the universal covers of $Z_{\mfr}$ (cf. Lemma~\ref{coro: fundamental group of Zm}).
\end{rmk}

\begin{theorem} \label{thm: Legendrian then scheme-theoretic intersection}
    If $\gfr$ is simple and $Z_{\mfr} = Z_{\text{long}}$, then the following are equivalent:
            \begin{enumerate}
                \item \label{cond: legendrian Om} $O_{\mfr}$ is a Legendrian subvariety of $Z_{\text{long}}$;
                \item \label{cond: clean intersection along Om} for each $x \in O_{\mfr}$, $T_{x} Z_{\text{long}} \cap T_{x} \PP(\mfr) = T_{x} O_{\mfr}$ in $T_{x} \PP(\gfr)$; and
                \item \label{cond: scheme theoretic intersection is Om} $O_{\mfr}$ is the scheme-theoretic intersection $Z_{\text{long}} \cap_{\text{sch}} \PP(\mfr)$ in $\PP(\gfr)$.
                That is, the ideal sheaf of $O_{\mfr}$ is the sum of the ideal sheaves of $Z_{\text{long}}$ and $\PP(\mfr)$ in $\PP(\gfr)$.
        \end{enumerate}
\end{theorem}

\begin{proof}
    Assume that $Z_{\mfr} = Z_{\text{long}}$, i.e., $O_{\mfr} \subset Z_{\text{long}}$.
    The condition (\ref{cond: scheme theoretic intersection is Om}) implies the condition (\ref{cond: clean intersection along Om}) by \cite[Lemma 5.1]{Li2009WonderfulCompactification}.
    To see the converse implication (\ref{cond: clean intersection along Om}) $\Rightarrow$ (\ref{cond: scheme theoretic intersection is Om}), by the same lemma, it suffices to show that the condition (\ref{cond: clean intersection along Om}) implies $O_{\mfr} = Z_{\text{long}} \cap \PP(\mfr)$ set-theoretically.
    Consider the inequalities
    \[
        \dim (T_{[v_{\rho}]} Z_{\text{long}} \cap T_{[v_{\rho}]} \PP(\mfr)) \ge \dim_{[v_{\rho}]} (Z_{\text{long}} \cap \PP(\mfr)) \ge \dim O_{\mfr}
    \]
    where $\dim_{[v_{\rho}]}(Z_{\text{long}} \cap \PP(\mfr))$ denotes the maximum among dimensions of irreducible components of $Z_{\text{long}} \cap \PP(\mfr)$ containing $[v_{\rho}]$.
    Since $O_{\mfr}$ is a unique closed $H$-orbit in $\PP(\mfr)$ and $Z_{\text{long}}$ is compact, $O_{\mfr}$ is contained in every irreducible component of $Z_{\text{long}} \cap \PP(\mfr)$, and hence $\dim_{[v_{\rho}]}(Z_{\text{long}} \cap \PP(\mfr)) = \dim(Z_{\text{long}} \cap \PP(\mfr))$.
    Therefore if the condition (\ref{cond: clean intersection along Om}) holds, then $\dim O_{\mfr} = \dim(Z_{\text{long}} \cap \PP(\mfr))$.
    Since $O_{\mfr}$ is compact and contained in every irreducible component of $Z_{\text{long}} \cap \PP(\mfr)$, we see that $O_{\mfr} = Z_{\text{long}} \cap \PP(\mfr)$ set-theoretically, and hence (\ref{cond: clean intersection along Om}) $\Rightarrow$ (\ref{cond: scheme theoretic intersection is Om}).

    Next, we show the equivalence (\ref{cond: legendrian Om}) $\Leftrightarrow$ (\ref{cond: clean intersection along Om}).
    If $\rho$ is a root of $\hfr$, then by Corollary \ref{coro: corollary of table, characterization of B3 G2}, Proposition \ref{prop: highest weight orbits for symm iso irre pair} and Proposition \ref{prop: rho is a highest root sp}, $Z_{\mfr}\not=Z_{\text{long}}$, a contradiction.
    Thus $\rho$ is not a root of $\hfr$, and so the equivalence (\ref{cond: legendrian Om}) $\Leftrightarrow$ (\ref{cond: clean intersection along Om}) follows from the inclusion $T_{[v_{\rho}]} Z_{\text{long}} \cap T_{[v_{\rho}]} \PP(\mfr) \supset T_{[v_{\rho}]} O_{\mfr}$, the inequality
    \[
        \dim Z_{\text{long}} + \dim \PP(\mfr) - \dim (T_{[v_{\rho}]} Z_{\text{long}} + T_{[v_{\rho}]} \PP(\mfr)) = \dim (T_{[v_{\rho}]} Z_{\text{long}} \cap T_{[v_{\rho}]} \PP(\mfr)) \ge \dim O_{\mfr},
    \]
    and the following lemma.
\end{proof}

\begin{lemma} \label{lem: dimension of sum of tangent sp of Z and Pm}
    If $\rho \not\in R_{\hfr}$, then
    \[
        \dim (T_{[v_{\rho}]} Z_{\mfr} + T_{[v_{\rho}]} \PP(\mfr)) = \dim \mfr + \dim O_{\mfr}
    \]
    where the sum of the tangent spaces is taken in $T_{[v_{\rho}]} \PP(\gfr)$.
\end{lemma}
\begin{proof}
    If we identify $T_{[v_{\rho}]} \PP(\gfr) \simeq \gfr / \mfr_{\rho}$, then
    \[
        T_{[v_{\rho}]} Z_{\mfr} + T_{[v_{\rho}]} \PP(\mfr) = ([\gfr,\, v_{\rho}] + \mfr)/\mfr_{\rho} = \left(\sum_{w \in W}[\mfr_{w},\, v_{\rho}] + \mfr\right) / \mfr_{\rho}
    \]
    since $[\hfr,\, v_{\rho}] \subset \mfr$ and $[\mfr_{0},\, \mfr_{\rho}] \subset \mfr_{\rho}$ (as $\rho$ is not a root of $\hfr$).
    For a weight vector $v_{w} \in \mfr_{w}$, if $w = -\rho$, then the $\hfr$-component of $[v_{-\rho},\, v_{\rho}]$ is nonzero and spans $\CC \cdot h_{\rho}$ where $h_{\rho}\in\tfr_{H}$ is the $b|_{\tfr_{H}}$-dual of $\rho$, i.e., $b(h_{\rho},\, -) = \rho(-)$ on $\tfr_{H}$, for the Killing form $b$ of $\gfr$.
    If $w \not= -\rho$, then the $\hfr$-component of $[v_{w},\,v_{\rho}]$ is contained in $\bigoplus_{\alpha \in R^{+}_{\hfr} : \alpha - \rho \in W} \hfr_{\alpha}$, and hence
    \[
        \sum_{w \in W}[\mfr_{w},\, v_{\rho}] + \mfr \subset \CC \cdot h_{\rho} \oplus \bigoplus_{\alpha \in R^{+}_{\hfr} : \alpha - \rho \in W} \hfr_{\alpha} \oplus \mfr.
    \]
    
    Observe that for $\alpha \in R_{\hfr}^{+}$, since $\rho + \alpha \not\in W$, $\rho - \alpha \in W$ (equivalently, $\alpha - \rho \in W$ as $\mfr$ is self-dual) if and only if $\alpha$ is not orthogonal to $\rho$.
    Thus the number of $\alpha \in R^{+}_{\hfr}$ such that $\rho - \alpha \in W$ is equal to $\dim O_{\mfr}$.
    Furthermore, for $\alpha \in R^{+}_{\hfr}$ satisfying $\rho - \alpha \in W$, we have $[v_{\rho}, \, E_{-\alpha}] \not= 0$, and so there is $v_{\alpha - \rho} \in \mfr_{\alpha - \rho}$ such that the $\hfr$-component of $[v_{\alpha - \rho}, \, v_{\rho}]$ is nonzero, since
    \[
        b([v_{\alpha-\rho}, \, v_{\rho}], \, E_{-\alpha}) = b(v_{\alpha-\rho}, \, [v_{\rho}, \, E_{-\alpha}])
    \]
    and since $b$ is non-degenerate.
    Therefore $\bigoplus_{\alpha\in R^{+}_{\hfr}: \alpha - \rho \in W} \hfr_{\alpha}\subset \sum_{w \in W}[\mfr_{w},\, v_{\rho}] + \mfr$, and hence
    \[
        \sum_{w \in W}[\mfr_{w},\, v_{\rho}] + \mfr = \CC \cdot h_{\rho} \oplus \bigoplus_{\alpha \in R^{+}_{\hfr} : \alpha - \rho \in W} \hfr_{\alpha} \oplus \mfr.
    \]

    To summarize, we have
    \[
        T_{[v_{\rho}]} Z_{\mfr} + T_{[v_{\rho}]} \PP(\mfr) = \left(\CC \cdot h_{\rho} \oplus \bigoplus_{\alpha \in R^{+}_{\hfr} : \alpha - \rho \in W} \hfr_{\alpha} \oplus \mfr \right) / \mfr_{\rho},
    \]
    and its dimension is equal to
    \[
        (1 + \dim O_{\mfr} + \dim \mfr) - 1.
    \]
\end{proof}

\begin{rmk} \label{rmk: linear section}
    \begin{enumerate}
        \item Theorem \ref{thm: Legendrian then scheme-theoretic intersection} implies that Legendrian subvarieties of adjoint varieties in Theorem \ref{main thm: adjoint} not arising from symmetric subalgebras of Hermitian type (i.e., not obtained by Proposition \ref{prop: Legendrian for IHSS}) are scheme-theoretic linear sections.

        \item Conversely, Buczy\'nski \cite[Corollary E.24]{Buczynski2008AlgebraicLegendrian} obtains the following result: if a smooth projective Legendrian subvariety of an adjoint variety is a scheme-theoretic linear section, then it is necessarily homogeneous under the action of its stabilizer in $S_{\text{ad}}$.
    \end{enumerate}
\end{rmk}

\section{Proof of Theorems \ref{main thm: adjoint}--\ref{main thm: semi simple} and Corollaries} \label{section: proof of classification}

    Finally, we complete the proof of Theorems \ref{main thm: adjoint}--\ref{main thm: semi simple}.

    \begin{proof}[Proof of Theorems \ref{main thm: adjoint}--\ref{main thm: semi simple}]
        First, the items in Theorems \ref{main thm: adjoint}--\ref{main thm: semi simple} indeed give us Legendrian subvarieties $O$ of nilpotent orbits $Z \subset \PP(\sfr)$ by the results of Section \ref{section: Structure of Isotropy Representation}: for Theorems \ref{main thm: adjoint}(1) and \ref{main thm: semi simple}(2.a--f), see Propositions \ref{prop: Legendrian for IHSS}, \ref{prop:Om when g is not simple} and \ref{prop: highest weight orbits for symm iso irre pair}; for Theorems \ref{main thm: adjoint}(2) and \ref{main thm: semi simple}(2.g), see Proposition \ref{prop: rho is a highest root sp}.

    For the converse, as in the statement, let $\tilde{\sfr}$ be a semi-simple Lie algebra, $Z$ a nilpotent orbit, and $\sfr$ the smallest ideal of $\tilde{\sfr}$ such that $Z \subset \PP(\sfr)$.
    (For example, if $\tilde{\sfr}$ is simple as in Theorem \ref{main thm: adjoint}, then $\tilde{\sfr} = \sfr$.)
    Assume that $O$ is a projective Legendrian subvariety of $Z$, and that $O$ is $\text{Stab}_{\tilde{S}_{\text{ad}}}(O)$-homogeneous.
    If $\tilde{S}_{\text{ad}}$ and $S_{\text{ad}}$ are the adjoint groups of $\tilde{\sfr}$ and $\sfr$, respectively, then we can write $\tilde{S}_{\text{ad}} = S_{\text{ad}} \times S'$ where $S'$ is the product of the remaining simple factors of $\tilde{S}_{\text{ad}}$.
    Since $S'$ acts trivially on $\sfr$, $Z$ is an $S_{\text{ad}}$-orbit in $\PP(\sfr)$ and $O$ is homogeneous under the action of $\text{Stab}_{S_{\text{ad}}}(O)$.
    Observe that the contact structures on $Z$ as a nilpotent orbit in $\PP(\tilde{\sfr})$ and as a nilpotent orbit in $\PP(\sfr)$ coincide, since $\sfr$ is an ideal of $\tilde{\sfr}$ and so the Killing forms of $\tilde{\sfr}$ and $\sfr$ are related via $b_{\tilde{\sfr}}|_{\sfr} = b_{\sfr}$.

    Next, by applying Theorem \ref{thm: iso irre var as Legendrian moduli} to the triple $O \subset Z \subset \PP(\sfr)$, we see that $S_{\text{ad}}$ acts on $M:=S_{\text{ad}}/\text{Stab}_{S_{\text{ad}}}(O)$ effectively, $\sfr/\ofr$ is an irreducible $\ofr^{\text{Levi}}$-representation, and $O \subset \PP(V)$ where $\ofr := T_{e} \text{Stab}_{S_{\text{ad}}}(O)$ and $V := \ofr^{\perp} = \{x \in \sfr : b_{\sfr}(\ofr,\, x) = 0\}$.
    In particular, since $\ofr$ is a maximal subalgebra of $\sfr$ (by the irreducibility of $\sfr/\ofr$), there are two possibilities: $\ofr$ is either reductive or parabolic (cf. \cite[Corollaire~1, \S10, Ch.~VIII]{Bourbaki}; see also \cite[Theorem, \S30.4, Ch.~X]{Humphreys}).
    \begin{itemize}
        \item If $\ofr$ is reductive, i.e., $\ofr = \ofr^{\text{Levi}}$, then $(\sfr,\,\ofr)$ is an isotropy irreducible pair, and $M$ is an isotropy irreducible variety of type $(\sfr,\, \ofr)$.
        Moreover, $V(\simeq \sfr/\ofr)$ is an irreducible $\text{Stab}_{S_{\text{ad}}}(O)$-representation, and hence $O \subset \PP(V)$ is its unique closed orbit (that is, $O_{\mfr} \subset \PP(\mfr)$ for $(\sfr,\,\ofr)$).
        Thus $O \subset Z$ can be obtained by Propositions \ref{prop:Om when g is not simple}, \ref{prop: highest weight orbits for symm iso irre pair}, and \ref{prop: rho is a highest root sp}, and all the possible cases are listed in Theorems \ref{main thm: adjoint}--\ref{main thm: semi simple}.
    
        \item If $\ofr$ is parabolic, then $M$ is a projective rational homogeneous space.
        Assume that $\sfr$ is not simple.
        Since $S_{\text{ad}}$ acts on $M$ effectively, $M$ splits into smaller rational homogeneous spaces, say $M \simeq \prod_{i} M_{i}$.
        It is a contradiction, since $\sfr/\ofr \simeq T_{[O]} M \simeq \bigoplus_{i} T_{[O]}M_{i}$ is not an irreducible $\ofr$-representation.
        Thus $\sfr$ is simple, and the irreducibility of $\sfr/\ofr$ implies that $M$ is an irreducible Hermitian symmetric space under the $S_{\text{ad}}$-action (cf. \cite[\S3.1]{LandsbergManivel2003ProjectiveGeometry}).
        Therefore $O \subset Z$ can be obtained by Proposition \ref{prop: Legendrian for IHSS}, which is belonging to Theorem \ref{main thm: adjoint}(1).
    \end{itemize}
    \end{proof}

\begin{exa}
    Recall that if $\sfr$ is of type $A$, then $Z_{\text{long}} \simeq \PP T^{*} \PP^{n+1}$, and every Legendrian submanifold is of form $\PP N^{*}_{Y/\PP^{n+1}}$.
    Using Theorem \ref{main thm: adjoint}, one may show that the images of equivariant Legendrian embeddings of rational homogeneous spaces into $\PP T^{*} \PP^{n+1}$ are given as $\PP N^{*}_{Y/\PP^{n+1}}$ where $Y \subset \PP^{n+1}$ is the highest weight orbit associated to one of the following, where the index means the corresponding item in Theorem \ref{main thm: adjoint}:
\begin{description}
    \item[(1)] the standard $\mathfrak{sl}(l)$-representation ($1 \le l \le n+1$), and hence $Y \subset \PP^{n+1}$ is a linear subspace $\PP^{l-1}$; \\
    \item[(1)] the standard $\mathfrak{so}(n+2)$-representation, and hence $Y \subset \PP^{n+1}$ is a quadric hypersurface $\QQ^{n}$; \\
    \item[(2.g)] ${\dynkin[labels={1}] A{x*.*} \text{($l-1$ nodes)} \otimes \dynkin[labels={1}] A{x}}$, and hence $Y \subset \PP^{n+1}=\PP^{2l-1}$ is the Segre embedding of $\PP^{l-1} \times \PP^{1}$; \\
    \item[(2.h)] ${\dynkin[labels={,,,,1}] D{****x}}$, and hence $Y \subset \PP^{n+1}=\PP^{15}$ is the Spinor variety $\SS_{5}$; and \\
    \item[(2.i)] ${\dynkin[labels={,1}] A{*x**}}$, and hence $Y \subset \PP^{n+1}=\PP^{9}$ is the Pl\"ucker embedding of $\text{Gr}(2,\, 5)$.
\end{description}

\end{exa}

Next, we prove two corollaries presented in the introduction.
In the following, as before, for a Lie group $R$, $R^{0}$ denotes the identity component.

    \begin{coro} \label{coro: not IHSS but Legendrian}
    There exist triples $O \hookrightarrow Z \subset \PP(\sfr)$ such that
    \begin{enumerate}
        \item $O$ is a projective Legendrian subvariety of a nilpotent orbit $Z \subset \PP(\sfr)$ such that $O$ is $\text{Stab}_{S_{\text{ad}}}(O)$-homogeneous;
        \item $O$ is indecomposable as a rational homogeneous space; and
        \item $O$ is not Hermitian symmetric with respect to the $\text{Aut}(O)^{0}$-action.
    \end{enumerate}
    The following list is a complete list of such triples  $O \hookrightarrow Z \subset \PP(\sfr)$ such that no proper ideal $\ifr < \sfr$ satisfies $Z \subset \PP(\ifr)$:
        \begin{itemize}
    \item (Orthogonal partial flag variety) $\text{OFl}(4,\, 5;\, \CC^{10}) \hookrightarrow Z_{\text{long}} \subset \PP(\mathfrak{sl}(16))$.

    \item (Partial flag variety) $\text{Fl}(2,\,3;\, \CC^{5}) \hookrightarrow Z_{\text{long}} \subset \PP(\mathfrak{sl}(10))$.

    \item (Orthogonal Grassmannian) $\text{OG}(3,\, \CC^{9}) \hookrightarrow Z_{\text{long}} \subset \PP(\mathfrak{so}(16))$.

    \item (Isotropic Grassmannian) $\text{IG}(2,\, \CC^{2l}) \hookrightarrow Z_{[2^{2},\, 1^{2l-4}]} \subset \PP(\mathfrak{sl}(2l))$, $l > 2$.

    \item $F_{4}/P_{1} \hookrightarrow Z_{2A_{1}} \subset \PP(E_{6})$.

    \item $Z_{\text{long}} \hookrightarrow \PP(\Ocal_{\text{min}}\oplus \Ocal_{\text{min}}) \subset \PP(\hfr' \oplus \hfr')$ for a simple Lie algebra $\hfr'$ not of type $C$.
    \end{itemize}
    \end{coro}
    \begin{proof}
        It is a direct consequence of Theorems \ref{main thm: adjoint}--\ref{main thm: semi simple}.
        Namely, see the following:
        \begin{itemize}
    \item Theorem \ref{main thm: adjoint}(2.h).

    \item Theorem \ref{main thm: adjoint}(2.i).

    \item Theorem \ref{main thm: adjoint}(2.j).

    \item Theorem \ref{main thm: semi simple}(2.c).

    \item Theorem \ref{main thm: semi simple}(2.f).

    \item Theorem \ref{main thm: semi simple}(2.a).
\end{itemize}
    \end{proof}

As the last application, we consider the case where the action of $\text{Aut}(O)^{0}$ does not extend to the $S_{\text{ad}}$-action.

\begin{coro} \label{coro: extend aut}
        Denote by $S_{\text{sc}}$ the simply connected Lie group associated to $\sfr$.
        Let $Z \subset \PP(\sfr)$ be a nilpotent orbit such that there is no proper ideal $\ifr < \sfr$ satisfying $Z \subset \PP(\ifr)$.
        Assume $O$ is a projective Legendrian subvariety of $Z$ such that $O$ is $\text{Stab}_{S_{\text{ad}}}(O)$-homogeneous but the restriction $\text{Stab}_{S_{\text{ad}}}(O)^{0} \rightarrow \text{Aut}(O)^{0}$ is not surjective.
        There exists a simple Lie algebra $\tilde{\sfr}$ containing $\sfr$ as a subalgebra satisfying the following conditions:
        \begin{enumerate}
            \item the adjoint variety $\tilde{Z}_{\text{long}} \subset \PP(\tilde{\sfr})$ contains an open $S_{\text{sc}}$-orbit $Z^{\text{sc}} \subset \tilde{Z}_{\text{long}}$ such that the orthogonal projection $\tilde{\sfr} \twoheadrightarrow \sfr$ induces an $S_{\text{sc}}$-equivariant universal covering $\varphi : Z^{\text{sc}} \rightarrow Z$ that is 2-to-1 and preserves the contact structures;
            \item $O$ admits an embedding $\iota: O \hookrightarrow Z^{\text{sc}}$ such that $\iota(O)$ is a Legendrian subvariety of $\tilde{Z}_{\text{long}}$ that is $\text{Stab}_{\tilde{S}_{\text{ad}}}(\iota(O))$-homogeneous, and the following diagram is commutative:
            \begin{center}
                \begin{tikzcd}
                    & Z^{\text{sc}} \arrow[r, phantom, sloped, "\subset"] \arrow[d, two heads, "\varphi"] & \tilde{Z}_{\text{long}} \arrow[r, phantom, sloped, "\subset"] & \PP(\tilde{\sfr}) \\
                   O \arrow[ru, hook, "\iota"] \arrow[r, phantom, sloped, "\subset"] & Z \arrow[r, phantom, sloped, "\subset"] & \PP(\sfr); &
                \end{tikzcd}
            \end{center}
            and
            \item if $\tilde{S}_{\text{ad}}$ is the adjoint group of $\tilde{\sfr}$, then the restriction $\text{Stab}_{\tilde{S}_{\text{ad}}}(\iota(O))^{0} \rightarrow \text{Aut}(O)^{0}$ is surjective.
        \end{enumerate}
        Moreover, the quadruple $O \hookrightarrow Z \subset \PP(\sfr) \hookrightarrow \PP(\tilde{\sfr})$ is one of the following:
        \begin{itemize}
            \item $\PP^{2l-1} \hookrightarrow \PP(\Ocal_{\text{min}} \oplus \Ocal_{\text{min}}) \subset \PP(C_{l} \oplus C_{l}) \hookrightarrow \PP(C_{2l})$, $l\ge2$.
            In this case, $(\sfr,\,Z,\,O)$ is from Theorem \ref{main thm: semi simple}(2.a) with $\hfr' = C_{l}$, while $(\tilde{\sfr},\, \tilde{Z}_{\text{long}},\, \iota(O))$ is from Theorem \ref{main thm: adjoint}(1) and Proposition \ref{prop: Legendrian for IHSS} with $S_{\text{ad}}/P = \text{LG}(2l,\, \CC^{4l})$;
            
            \item $\PP^{2p-1} \times \PP^{2(l-p)-1} \hookrightarrow Z_{\text{short}} \subset \PP(C_{l}) \hookrightarrow \PP(A_{2l-1})$, $1 \le p \le l-1$ but $(p,\, l)\not= (1,\,2)$.
            In this case, $(\sfr,\,Z,\,O)$ is from Theorem \ref{main thm: semi simple}(2.b), while $(\tilde{\sfr},\, \tilde{Z}_{\text{long}},\, \iota(O))$ is from Theorem \ref{main thm: adjoint}(1) and Proposition \ref{prop: Legendrian for IHSS} with $S_{\text{ad}}/P = \text{Gr}(2p,\, \CC^{2l-1})$;
            
            \item $\text{OG}(4,\, \CC^{9}) \hookrightarrow Z_{\text{short}} \subset \PP(F_{4}) \hookrightarrow \PP(E_{6})$.
            In this case, $(\sfr,\,Z,\,O)$ is from Theorem \ref{main thm: semi simple}(2.e), while $(\tilde{\sfr},\, \tilde{Z}_{\text{long}},\, \iota(O))$ is from Theorem \ref{main thm: adjoint}(1) and Proposition \ref{prop: Legendrian for IHSS} with $S_{\text{ad}}/P = \mathbb{OP}^{2}$; and
            
            \item $\QQ^{5} \hookrightarrow Z_{[3,\,2^{2}]} \subset \PP(B_{3}) \hookrightarrow \PP(B_{4})$.
            In this case, $(\sfr,\,Z,\,O)$ is from Theorem \ref{main thm: semi simple}(2.g), while $(\tilde{\sfr},\, \tilde{Z}_{\text{long}},\, \iota(O))$ is from Theorem \ref{main thm: adjoint}(1) and Proposition \ref{prop: Legendrian for IHSS} with $S_{\text{ad}}/P = \QQ^{7}$.
        \end{itemize}
    \end{coro}

To prove Corollary \ref{coro: extend aut}, we need to construct a universal cover of $Z$, which is done in the following Lemma \ref{coro: fundamental group of Zm}.
In the following, recall that $(\gfr,\,\hfr)$ is an isotropy irreducible pair with $\dim \gfr > 1$ and that $Z_{\mfr} \subset \PP(\gfr)$ is a nilpotent orbit containing $O_{\mfr} \subset \PP(\mfr)$, and denote by $G_{\text{sc}}$ the simply connected Lie group associated to $\gfr$.
    
\begin{lemma} \label{coro: fundamental group of Zm}
    If $Z_{\mfr}$ is not simply connected, then $(\gfr,\,\hfr)$ is one of ($C_{l} \oplus C_{l}$, diag($C_{l}$)) ($l\ge 1$), ($C_{l}$, $C_{p} \oplus C_{l-p}$) ($1 \le p \le l-1$), ($\mathfrak{so}(l)$, $\mathfrak{so}(l-1)$) ($l \ge 5$), ($F_{4}$, $B_{4}$) (symmetric), and ($B_{3}$, $G_{2}$) (non-symmetric).
    In this case, $\pi_{1}(Z_{\mfr}) = \ZZ/ 2 \ZZ$, and its universal cover $\varphi:Z^{\text{sc}}_{\mfr} \rightarrow Z_{\mfr}$ can be constructed as follows:
    $\gfr$ can be embedded into a simple Lie algebra $\tilde{\gfr}$ so that the adjoint variety $\tilde{Z}_{\text{long}} \subset \PP(\tilde{\gfr})$ contains an open $G_{\text{sc}}$-orbit $Z_{\mfr}^{\text{sc}}$, the covering $\varphi: Z^{\text{sc}}_{\mfr} \rightarrow Z_{\mfr}$ is the restriction of the orthogonal projection $\tilde{\gfr} \twoheadrightarrow \gfr$, and $d \varphi (\tilde{D}|_{Z_{\mfr}^{\text{sc}}}) = D$ where $\tilde{D}$ and $D$ are the contact structures of $\tilde{Z}_{\text{long}}$ and $Z_{\mfr}$, respectively.
\end{lemma}
\begin{proof}
    If $Z_{\mfr} = Z_{\text{long}}$, then it is a rational homogeneous space, and hence simply connected.
    By Proposition \ref{prop:Om when g is not simple} and Theorem \ref{thm: classification of Om and Zm of iso irre var}, it suffices to consider the following cases:
    \begin{itemize}
        \item ($\hfr' \oplus \hfr'$, diag($\hfr'$)) for a simple Lie algebra $\hfr'$:
        In this case, $Z_{\mfr} = \PP(\Ocal_{\text{min}} \oplus \Ocal_{\text{min}})$.
        If $\hfr'$ is not of type $C$, then $\Ocal_{\text{min}}$ is simply connected by \cite[Corollary 6.1.6 and \S 8.4]{CollingwoodMcGovern1993NilpotentOrbits}, and so is $Z_{\mfr}$.
        The case of type $C$ is considered below.
    
        \item ($A_{2l-1}$, $C_{l}$) ($l \ge 2$):
        In this case, $Z_{\mfr} = Z_{[2^{2},\, 1^{2l-4}]}$.
        If $l \ge 3$, then it is simply connected by \cite[Corollary 6.1.6]{CollingwoodMcGovern1993NilpotentOrbits}.
        If $l=2$, this pair coincides with ($D_{3}$, $B_{2}$) = ($\mathfrak{so}(6)$, $\mathfrak{so}(5)$), which is considered below.
        
        \item ($E_{6}$, $F_{4}$):
        In this case, $Z_{\mfr} = Z_{2A_{1}}$, which is simply connected by \cite[\S 8.4]{CollingwoodMcGovern1993NilpotentOrbits}.
        
        \item ($C_{l} \oplus C_{l}$, diag($C_{l}$)) ($l\ge 1$), ($C_{l}$, $C_{p} \oplus C_{l-p}$) ($1 \le p \le l-1$), ($\mathfrak{so}(l)$, $\mathfrak{so}(l-1)$) ($l \ge 5$), ($F_{4}$, $B_{4}$), and ($B_{3}$, $G_{2}$):
        In these cases, $Z_{\mfr}$ is $\PP(\Ocal_{\text{min}} \oplus \Ocal_{\text{min}})$, $Z_{\text{short}} (= Z_{[2^{2},\, 1^{2l-4}]})$, $Z_{[3,\, 1^{l-3}]}$, $Z_{\text{short}} (= Z_{\tilde{A}_{1}})$ and $Z_{[3,\,2^{2}]}$, respectively.
        Consider a simple Lie algebra
        \[
            \tilde{\gfr} := \left\{ \begin{array}{ll}
                C_{2l} & \text{if } (\gfr,\,\hfr) = (C_{l} \oplus C_{l},\, \text{diag}(C_{l})) \ (l \ge 1), \\
                A_{2l-1} & \text{if } (\gfr,\,\hfr) = (C_{l},\, C_{p} \oplus C_{l-p}) \ (1 \le p \le l-1), \\
                \mathfrak{so}(l+1) &  \text{if } (\gfr,\,\hfr) = (\mathfrak{so}(l),\, \mathfrak{so}(l-1)) \ (l \ge 5), \\
                E_{6} & \text{if } (\gfr,\,\hfr) = (F_{4},\, B_{4}),\\
                B_{4} & \text{if } (\gfr,\,\hfr) = (B_{3},\, G_{2}).
            \end{array} \right.
        \]
        Consider the embedding $\gfr \hookrightarrow \tilde{\gfr}$ given as follows:
        \begin{itemize}
            \item If $(\gfr,\,\hfr) = (C_{l} \oplus C_{l},\, \text{diag}(C_{l}))$, then $\gfr = \mathfrak{sp}(V) \oplus \mathfrak{sp}(V) \hookrightarrow \tilde{\gfr} = \mathfrak{sp}(V\oplus V)$ is the standard embedding where $V$ is a symplectic vector space of dimension $2l$.
            \item If $(\gfr,\,\hfr) = (\mathfrak{so}(l),\, \mathfrak{so}(l-1))$, then $\gfr \hookrightarrow \tilde{\gfr} = \mathfrak{so}(l+1)$ is the standard embedding.
            More precisely, the matrix algebra $\mathfrak{so}(l)$ of skew-symmetric $l \times l$ matrices is identified with the subalgebra of $\mathfrak{so}(l+1)$ consisting of skew-symmetric $(l+1) \times (l+1)$ matrices whose $(l+1)$-th rows and $(l+1)$-th columns are zero.
            \item If $(\gfr,\,\hfr) = (B_{3},\, G_{2})$, consider the composition $\gfr = B_{3} \hookrightarrow D_{4} \xrightarrow{\sigma} D_{4} \hookrightarrow \tilde{\gfr} = B_{4}$ of the standard embeddings $B_{3} \hookrightarrow D_{4}$ and $D_{4} \hookrightarrow B_{4}$, and the triality $\sigma$ ($=\sigma_{3}$ in the proof of Proposition \ref{prop: rho is a highest root sp}).
            \item In other cases, consider the embedding $\gfr \hookrightarrow \tilde{\gfr}$ as the fixed-point-locus of the involution of $\tilde{\gfr}$ induced by the diagram involution (see the proof of Proposition \ref{prop: highest weight orbits for symm iso irre pair}).
        \end{itemize}
        If $\Ocal \subset \gfr$ is the nilpotent orbit such that $\PP \Ocal = Z_{\mfr}$ and $\tilde{\Ocal}_{\text{min}}$ is the minimal nilpotent orbit in $\tilde{\gfr}$, then by \cite[Proposition 2.12]{FuJuteauLevySommers2023LocalGeometry}, the orthogonal projection $\tilde{\gfr} \twoheadrightarrow \gfr$ induces a finite $G_{\text{ad}}$-equivariant morphism $\overline{\tilde{\Ocal}_{\text{min}}} \twoheadrightarrow \overline{\Ocal}$ between the closures, which is 2-to-1 over $\Ocal$.
        This morphism induces a finite $G_{\text{sc}}$-equivariant morphism
        \[
            \varphi: \tilde{Z}_{\text{long}} (=\PP(\tilde{\Ocal}_{\text{min}})\subset \PP(\tilde{\gfr})) \twoheadrightarrow \overline{Z_{\mfr}} (= (\overline{\Ocal} \setminus \{0\} / \CC^{\times})\subset \PP(\gfr)),
        \]
        which is 2-to-1 over $Z_{\mfr}$.
        Furthermore, by the $G_{\text{sc}}$-equivariance, $\varphi^{-1}(Z_{\mfr})$ is an open $G_{\text{sc}}$-orbit in $\tilde{Z}_{\text{long}}$.
        Since nilpotent orbits are contact manifolds and $\overline{Z_{\mfr}}$ is a union of nilpotent orbits, the complement of $Z_{\mfr}$ in $\overline{Z_{\mfr}}$ is of (complex) codimension at least 2.
        Thus $\tilde{Z}_{\text{long}} \setminus \varphi^{-1}(Z_{\mfr})$ is of codimension at least 2 in $\tilde{Z}_{\text{long}}$, and hence $\varphi^{-1}(Z_{\mfr})$ is simply connected.
        It means that the restriction $Z^{\text{sc}}_{\mfr}:=\varphi^{-1}(Z_{\mfr}) \rightarrow Z_{\mfr}$ of $\varphi$ is a universal cover, and hence $\pi_{1}(Z_{\mfr}) = \ZZ / 2\ZZ$.

        To show the last property, choose points $[v] \in Z_{\mfr}$ and $[\tilde{v}] \in \varphi^{-1}([v])$.
        Up to scalar multiplication, we may choose $v$ and $\tilde{v}$ so that $\tilde{v} - v \in \gfr^{\perp}$, i.e., $b_{\tilde{\gfr}}(\gfr,\, \tilde{v} - v) = 0$.
        Recall that at the points $[\tilde{v}]$ and $[v]$, the contact structures $\tilde{D}$ and $D$ are given as
        \begin{align*}
            \tilde{D}_{[\tilde{v}]} &\simeq \{w \in \tilde{\gfr}: b_{\tilde{\gfr}}(w,\, \tilde{v}) = 0\} / \nfr_{\tilde{\gfr}}(\tilde{v}), \\
            D_{[v]} &\simeq \{w \in \gfr: b_{\gfr}(w,\, v) = 0\} / \nfr_{\gfr}(v),
        \end{align*}
        respectively.
        Also recall that if $\gfr$ is simple, then there exists a unique $G_{\text{sc}}$-invariant non-degenerate bilinear form up to scalar multiplication.
        Thus if $\gfr$ is simple, then $b_{\tilde{\gfr}}|_{\gfr} = c \cdot b_{\gfr}$ for some $c \in \CC^{\times}$.
        Moreover, if $\gfr = C_{l} \oplus C_{l} (= \mathfrak{sp}(V) \oplus \mathfrak{sp}(V))$ and $\tilde{\gfr} = C_{2l} (=\mathfrak{sp}(V\oplus V))$, then since the embedding $\gfr \hookrightarrow \tilde{\gfr}$ is the standard one, the restriction of $b_{\tilde{\gfr}}$ on each simple factor of $\gfr$ is of form $c' \cdot b_{C_{l}}$ for the same constant $c' \in \CC^{\times}$.
        Therefore in any case, we have $b_{\tilde{\gfr}}|_{\gfr} = c \cdot b_{\gfr}$, $c \in \CC^{\times}$.
        Finally, since $Z_{\mfr}^{\text{sc}}$ is an open $G_{\text{sc}}$-orbit in $\tilde{Z}_{\text{long}}$, every tangent vector of $\tilde{Z}_{\text{long}}$ at $[\tilde{v}]$ is represented by $w \in \gfr$, and it is an element of $\tilde{D}_{[\tilde{v}]}$ if and only if
        \[
            0 =b_{\tilde{\gfr}}(w,\, \tilde{v}) = b_{\tilde{\gfr}}(w,\, v) = c \cdot b_{\gfr}(w,\, v).
        \]
        Since $d_{[\tilde{v}]} \varphi (w \mod \nfr_{\tilde{\gfr}}(\tilde{v})) = w \mod \nfr_{\gfr}(v)$, we conclude that $d\varphi(\tilde{D}_{[\tilde{v}]}) = D_{[v]}$.
        Now the statement follows from the $G_{\text{sc}}$-equivariance.
    \end{itemize}
\end{proof}

    \begin{proof}[Proof of Corollary \ref{coro: extend aut}]
        It is a direct consequence of Theorems \ref{main thm: adjoint}--\ref{main thm: semi simple} that the triple $(\sfr,\,Z,\,O)$ satisfying the conditions in the statement is one of
        \begin{itemize}
            \item $(C_{l}\oplus C_{l},\, \PP(\Ocal_{\text{min}} \oplus \Ocal_{\text{min}}),\, \PP^{2l-1})$, $l \ge 2$ (Theorem \ref{main thm: semi simple}(2.a) with $\hfr' = C_{l}$).
            In this case, $\text{Stab}_{S_{\text{ad}}}(O)^{0}$ is of type $C_{l}$ while $\text{Aut}(\PP^{2l-1})^{0}$ is of type $A_{2l-1}$;
            
            \item $(C_{l},\, Z_{\text{short}},\, \PP^{2p-1} \times \PP^{2(l-p)-1})$, $1 \le p \le l-1$ but $(p,\,l)\not=(1,\,2)$ (Theorem \ref{main thm: semi simple}(2.b)).
            In this case, $\text{Stab}_{S_{\text{ad}}}(O)^{0}$ is of type $C_{p} \oplus C_{l-p}$ while $\text{Aut}(\PP^{2l-1})^{0}$ is of type $A_{2p-1} \oplus A_{2(l-p)-1}$;

            \item $(F_{4},\, Z_{\text{short}},\, \text{OG}(4,\, \CC^{9}))$ (Theorem \ref{main thm: semi simple}(2.e)).
            In this case, $\text{Stab}_{S_{\text{ad}}}(O)^{0}$ is of type $B_{4}$ while $\text{Aut}(\text{OG}(4,\, \CC^{9}))^{0}$ is of type $D_{5}$;

            \item $(B_{3},\, Z_{[3,\,2^{2}]},\, \QQ^{5})$ (Theorem \ref{main thm: semi simple}(2.g)).
            In this case, $\text{Stab}_{S_{\text{ad}}}(O)^{0}$ is of type $G_{2}$ while $\text{Aut}(\QQ^{5})^{0}$ is of type $B_{3}$;
        \end{itemize}

        Namely, these are defined by isotropy irreducible pairs $(\gfr,\,\hfr)$ such that $Z_{\mfr}$ is not simply connected, as in Lemma \ref{coro: fundamental group of Zm}.
        (More precisely, we take $(\gfr,\,\hfr)$ as $(C_{l}\oplus C_{l},\, \text{diag}(C_{l}))$, $(C_{l},\, C_{p} \oplus C_{l-p})$, $(F_{4},\, B_{4})$ and $(B_{3},\,G_{2})$, respectively.)
        As in the proof of Lemma \ref{coro: fundamental group of Zm}, we put $\tilde{\sfr}$ as $C_{2l}$, $A_{2l-1}$, $E_{6}$ and $B_{4}$, respectively, so that $\tilde{Z}_{\text{long}}$ contains an open $S_{\text{sc}}$-orbit $Z^{\text{sc}}$ equipped with a double cover $\varphi : Z^{\text{sc}} \rightarrow Z$ induced by the orthogonal projection $\tilde{\sfr} \twoheadrightarrow \sfr$.
        Since $O$ is simply connected (being a rational homogeneous space) and $\varphi$ is $S_{\text{sc}}$-equivariant, $\varphi^{-1}(O)$ consists of two connected components $O_{1}$ and $O_{2}$, and $\varphi$ maps each $O_{i}$ isomophically onto $O$.
        Moreover, since $\varphi$ preserves the contact structures and it is $S_{\text{sc}}$-equivariant, each $O_{i}$ is a Legendrian subvariety of $\tilde{Z}_{\text{long}}$ that is $\text{Stab}_{\tilde{S}_{\text{ad}}}(O_{i})$-homogeneous.
        It follows that the restriction $\text{Stab}_{\tilde{S}_{\text{ad}}}(O_{i})^{0} \rightarrow \text{Aut}(O_{i})^{0}$ is surjective, since none of the above triples $(\sfr,\,Z,\,O)$ satisfies $Z = Z_{\text{long}}$.
        Finally, observe that the new triples $(\tilde{\sfr},\, \tilde{Z}_{\text{long}},\, O_{i})$ can be obtained by Theorem \ref{main thm: adjoint} and Proposition \ref{prop: Legendrian for IHSS} with the irreducible Hermitian symmetric spaces indicated in the statement.
    \end{proof}

\section{Tables} \label{section: Tables}

In this section, four tables are given.
In Table \ref{table: classification of g h} (obtained in Theorem \ref{thm: classification of isotropy irre pair} and Proposition \ref{prop: rho is a highest root sp}), we recall the classification of non-symmetric $(\gfr,\, \hfr)$, together with dimension of $O_{\mfr}$ and $Z_{\mfr}$.
Namely, $O_{\mfr}$ corresponding to a row marked as {\lq Yes\rq} in the last column is a Legendrian subvariety in Theorems \ref{main thm: adjoint}(2) and \ref{main thm: semi simple}(2.g).
In Table \ref{table: isotropy irreducible pairs of equal rank}--\ref{table: Legendrian asso to IHSS} (obtained in Propositions \ref{prop: Legendrian for IHSS}, \ref{prop:Om when g is not simple}, \ref{prop: highest weight orbits for symm iso irre pair}), we collect the well-known classification of symmetric pairs and irreducible Hermitian symmetric spaces, together with Legendrian subvarieties in Theorems \ref{main thm: adjoint}(1) and \ref{main thm: semi simple}(2).

In the tables, we keep the notation of the previous sections. Namely, for a given reductive Lie algebra and its simple factor $\hfr_{1}$, we denote by $\alpha_{i}^{\hfr_{1}}$, $\pi_{i}^{\hfr_{1}}$, $\delta^{\hfr_{1}}$, and $\delta_{\text{short}}^{\hfr_{1}}$ a simple root, a fundamental weight, the highest root of $\hfr_{1}$, and the dominant short root, respectively.
We drop the superscript $\hfr_{1}$ if the given reductive Lie algebra is indeed simple.
For the indexing of the Dynkin diagrams, we follow the notation of \cite{OnishchikVinberg1990LieGroups}.
In our notation, $\delta$ and $\delta_{\text{short}}$ are given as follows:
\begin{itemize}
    \item $A_{r}$ ($r \ge 1$): $\delta = \alpha_{1} + \cdots + \alpha_{r}$ for the indexing
    \[
        {\dynkin[labels={\alpha_{1},\alpha_{2},\alpha_{r-1},\alpha_{r}}, edge length=.75cm] A{}}.
    \]
    \item $B_{r}$ ($r \ge 2$): $\delta = \alpha_{1} + 2 \alpha_{2} + \cdots + 2\alpha_{r}$ and $\delta_{\text{short}} = \alpha_{1} + \cdots + \alpha_{r}$ for the indexing
    \[
        {\dynkin[labels={\alpha_{1},\alpha_{2},\alpha_{r-2},\alpha_{r-1},\alpha_{r}}, edge length=.75cm] B{}}.
    \]
    \item $C_{r}$ ($r \ge 2$): $\delta = 2\alpha_{1} + \cdots + 2 \alpha_{r-1} + \alpha_{r}$ and $\delta_{\text{short}} = \alpha_{1} + 2 \alpha_{2} + \cdots + 2 \alpha_{r-1} + \alpha_{r}$ for the indexing
    \[
        {\dynkin[labels={\alpha_{1},\alpha_{2},\alpha_{r-2},\alpha_{r-1},\alpha_{r}}, edge length=.75cm] C{}}.
    \]
    \item $D_{r}$ ($r \ge 3$): $\delta = \alpha_{1} + 2 \alpha_{2} + \cdots + 2 \alpha_{r-2} + \alpha_{r-1} + \alpha_{r}$ for the indexing
    \[
        {\dynkin[labels={1, 2, r-3, r-2, r-1, r},label macro/.code={\alpha_{\drlap{#1}}},
        edge length=.75cm]D{}}.
    \]

    \item $G_{2}$: $\delta = 3\alpha_{1} + 2\alpha_{2}$ and $\delta_{\text{short}} = 2 \alpha_{1} + \alpha_{2}$ for the indexing
    \[
        {\dynkin[labels={1,2},label macro/.code={\alpha_{\drlap{#1}}},
        edge length=.75cm]G2}.
    \]

    \item $F_{4}$: $\delta = 2\alpha_{1} + 4\alpha_{2} + 3\alpha_{3} + 2\alpha_{4}$ and $\delta_{\text{short}}= 2 \alpha_{1} + 3 \alpha_{2} + 2 \alpha_{3} + \alpha_{4}$ for the indexing
    \[
        {\dynkin[labels={4,3,2,1},label macro/.code={\alpha_{\drlap{#1}}},
        edge length=.75cm, backwards]F4}.
    \]
    
    \item $E_{6}$: $\delta =  \alpha_{1} + 2 \alpha_{2} + 3 \alpha_{3} + 2 \alpha_{4} + \alpha_{5} + 2 \alpha_{6}$ for the indexing
    \[
        {\dynkin[labels={1, 6, 2, 3, 4, 5},label macro/.code={\alpha_{\drlap{#1}}},
        edge length=.75cm, upside down]E6}.
    \]
    \item $E_{7}$: $\delta = \alpha_{1} + 2 \alpha_{2} + 3 \alpha_{3} + 4 \alpha_{4} + 3 \alpha_{5} + 2 \alpha_{6} + 2 \alpha_{7}$ for the indexing
    \[
        {\dynkin[labels={6,7,5,4,3,2,1},label macro/.code={\alpha_{\drlap{#1}}},
        edge length=.75cm, backwards, upside down]E7}.
    \]
    \item $E_{8}$: $\delta = 2\alpha_{1} + 3\alpha_{2} + 4\alpha_{3} + 5\alpha_{4} + 6\alpha_{5} + 4\alpha_{6} + 2 \alpha_{7} + 3 \alpha_{8}$ for the indexing
    \[
        {\dynkin[labels={7,8,6,5,4,3,2,1},label macro/.code={\alpha_{\drlap{#1}}},
        edge length=.75cm, backwards, upside down]E8}.
    \]
    
\end{itemize}

\begin{rmk} \label{rmk:notation tables}
    \begin{enumerate}
    \item $D_{1} := \mathfrak{so}(2)$ is a 1-dimensional reductive Lie algebra.
    \item In the isomorphism $D_{2} := \mathfrak{so}(4) \simeq A_{1} \oplus A_{1}$, the simple factors are written as $A_{1}'$ and $A_{1}''$.
    \item In the tables of \cite[\S I.11]{Wolf1968GeometryStructure} and \cite{Wolf1984CorrectionGeometry}, our non-symmetric isotropy irreducible pairs $(\gfr,\hfr)$ are corresponding to (the complexifications of) the rows whose isotropy representations are absolutely irreducible, i.e., the rows with connected diagrams in the column $\chi$.
    In the same tables, the embedding of $\hfr$ into $\gfr$ is also described, in the column $\pi$.
\end{enumerate}
\end{rmk}

\begin{table}[p]
    \centering
    \hspace*{-.2\textwidth} 
    \resizebox{1.3\textwidth}{!}{%
    \begin{tabular}{|c||c|c|c|c|c|}
    \hline
        No. & ($\gfr$, $\hfr$) & Highest weight $\rho$ of $\mfr$ & $\dim O_{\mfr}$ & $\dim Z_{\mfr}$ & Legendrian? \\ \hline \hline
        1$_{p,\,q}$ ($p \ge q \ge 2$, $pq > 4$) & ($A_{pq-1}$, $A_{p-1} \oplus A_{q-1}$) & $\pi_{1}^{A_{p-1}} + \pi_{p-1}^{A_{p-1}} + \pi_{1}^{A_{q-1}} + \pi_{q-1}^{A_{q-1}} = \delta^{A_{p-1}} + \delta^{A_{q-1}}$ & $2p+2q-6$ & $2pq-3$ & Yes if $q=2$ \\ 
        2 & ($A_{15}$, $D_{5}$) & $\pi_{4} + \pi_{5} = \delta + \alpha_{3} + \alpha_{4} + \alpha_{5}$ & 14 & 29 & Yes \\ 
        3 & ($A_{26}$, $E_{6}$)& $\pi_{1} + \pi_{5} = \delta + \alpha_{1} + \alpha_{2} + \alpha_{3} + \alpha_{4} + \alpha_{5}$ & 24 & 51 & \\ 
        4$_{n}$ ($n \ge 5$) & ($A_{n(n-1)/2-1}$, $A_{n-1}$) & $\pi_{2}+\pi_{n-2} = \delta + \alpha_{2} + \cdots + \alpha_{n-2}$ & $4n-12$ & $n^{2}-n-3$ & Yes if $n=5$ \\ 
        5$_{n}$ ($n \ge 3$) & ($A_{n(n+1)/2-1}$, $A_{n-1}$) & $2 \pi_{1} + 2 \pi_{n-1} = 2 \delta$ & $2n-3$ & $n^{2}+n-3$ & \\ 
        6 & ($C_{2}$, $A_{1}$) & $6 \pi_{1} = 3 \delta$ & 1 & 3 & Yes \\ 
        7 & ($C_{7}$, $C_{3}$) & $2 \pi_{3} = \delta + 2 \alpha_{2} + 2 \alpha_{3}$ & 6 & 13 & Yes \\ 
        8 & ($C_{10}$, $A_{5}$) & $2 \pi_{3}= \delta + \alpha_{2}+2\alpha_{3}+\alpha_{4}$ & 9 & 19 & Yes \\ 
        9 & ($C_{16}$, $D_{6}$) & $2 \pi_{5} = \delta + \alpha_{3} + 2 \alpha_{4}+2 \alpha_{5} + \alpha_{6}$ & 15 & 31 & Yes \\ 
        10 & ($C_{28}$, $E_{7}$)& $2 \pi_{1} = \delta + 2 \alpha_{1} + 2 \alpha_{2} + 2 \alpha_{3} + 2 \alpha_{4} + \alpha_{5} + \alpha_{7}$ & 27 & 55 & Yes \\ 
        11$_{n}$ ($n \ge 3$) & ($C_{n}$, $A_{1} \oplus \mathfrak{so}(n)$)  & (if $n=3$) $2 \pi_{1}^{A_{1}} + 4 \pi_{1}^{\mathfrak{so}(3)} = \delta^{A_{1}} + 2 \delta^{\mathfrak{so}(3)}$ &$n-1$ & $2n-1$ & Yes \\ 
        && (if $n=4$) $2 \pi_{1}^{A_{1}} + 2 \pi_{1}^{A_{1}'} + 2 \pi_{1}^{A_{1}''} = \delta^{A_{1}} + \delta^{A_{1}'} + \delta^{A_{1}''}$ & && \\
        && (if $n \ge 5$) $2 \pi_{1}^{A_{1}} + 2 \pi_{1}^{\mathfrak{so}(n)} = \delta^{A_{1}} + \delta^{\mathfrak{so}(n)} + \alpha_{1}^{\mathfrak{so}(n)}$ &&&\\ 
        12 & ($D_{10}$, $A_{3}$) & $\pi_{1} + 2 \pi_{2} + \pi_{3}=2 \delta + \alpha_{2}$ &6 & 33 & \\ 
        13& ($D_{35}$, $A_{7}$)& $\pi_{3} + \pi_{5}=\delta + \alpha_{2} + 2 \alpha_{3} + 2 \alpha_{4} + 2 \alpha_{5} + \alpha_{6}$ &21 &133 & \\ 
        14 & ($D_{8}$, $B_{4}$)& $\pi_{3}=\delta + \alpha_{3} + \alpha_{4}$ & 12 & 25 & Yes \\ 
        15$_{n}$ ($n \ge 2$) & ($\mathfrak{so}(2n^{2}+n)$, $B_{n}$) & (if $n=2$) $\pi_{1}+2 \pi_{2} = 2\delta - \alpha_{2}$ & (if $n=2$) 4& $4n^{2}+2n-7$ & \\
        && (if $n=3$) $\pi_{1} + 2 \pi_{3} = 2\delta - \alpha_{2}$ & (if $n \ge 3$) $6n-10$& & \\ 
        && (if $n\ge 4$) $\pi_{1} + \pi_{3} = 2\delta - \alpha_{2}$ & && \\ 
        16$_{n}$ ($n \ge 2$) & ($\mathfrak{so}(2n^{2}+3n)$, $B_{n}$) & (if $n=2$) $2 \pi_{1} + 2\pi_{2} = 2 \delta + \alpha_{1}$ & $4n-4$ & $4n^{2}+6n-7$ & \\ 
        && (if $n \ge 3$) $2 \pi_{1} + \pi_{2} = 2 \delta + \alpha_{1}$ &&& \\ 
        17 & ($D_{21}$, $C_{4}$)& $2\pi_{3} = \delta + 2 \alpha_{2} + 4 \alpha_{3} + 2 \alpha_{4}$ & 12 & 77 & \\ 
        18$_{n}$ ($n \ge 3$) & ($\mathfrak{so}(2n^{2}-n-1)$, $C_{n}$) & $\pi_{1}+\pi_{3} = \delta + \alpha_{2} + 2\alpha_{3} + \cdots + 2 \alpha_{n-1} + \alpha_{n}$ & $6n-10$ & $4n^{2}-2n-9$ & \\ 
        19$_{n}$ ($n \ge 3$) & ($\mathfrak{so}(2n^{2}+n)$, $C_{n}$) & $2\pi_{1}+\pi_{2} = 2\delta - \alpha_{1}$ & $4n-4$ & $4n^{2}+2n-7$ & \\ 
        20 & ($D_{64}$, $D_{8}$)& $\pi_{6} = \delta + \alpha_{3} + 2 \alpha_{4} + 3 \alpha_{5} + 4 \alpha_{6} + 2 \alpha_{7} + 2 \alpha_{8}$ & 39 & 249 & \\ 
        21$_{n}$ ($n \ge 4$) & ($\mathfrak{so}(2n^{2}-n)$, $D_{n}$) & (if $n=4$) $\pi_{1}+\pi_{3}+\pi_{4} = 2\delta - \alpha_{2}$ & $6n-13$ & $4n^{2}-2n-7$ & \\ 
        && (if $n \ge 5$) $\pi_{1}+\pi_{3} = 2\delta - \alpha_{2}$ & && \\ 
        22$_{n}$ ($n \ge 4$) & ($\mathfrak{so}(2n^{2}+n-1)$, $D_{n}$) & $2 \pi_{1}+\pi_{2} = 2 \delta + \alpha_{1}$ & $4n-6$ & $4n^{2}+2n-9$ & \\ 
        23 & ($B_{3}$, $G_{2}$) & $\pi_{1} = \delta_{\text{short}}$ & 5 & 11 & Yes \\ 
        24 & ($D_{7}$, $G_{2}$)& $3\pi_{1} = 2\delta - \alpha_{2}$ & 5 & 21 & \\ 
        25 & ($D_{13}$, $F_{4}$)& $\pi_{2}=\delta + \alpha_{1}+2\alpha_{2}+\alpha_{3}$ & 20 & 45 & \\ 
        26 & ($D_{26}$, $F_{4}$)& $\pi_{3} = 2\delta-\alpha_{4}$ & 20 & 97 & \\ 
        27 & ($D_{39}$, $E_{6}$)& $\pi_{3} = 2\delta - \alpha_{6}$ & 29 & 149 & \\ 
        28 & ($B_{66}$, $E_{7}$)& $\pi_{5} = 2\delta - \alpha_{6}$ & 47 & 259 & \\ 
        29 & ($D_{124}$, $E_{8}$)& $\pi_{2} = 2\delta - \alpha_{1}$ & 83 & 489 & \\ 
        30$_{n}$ ($n \ge 3$) & ($D_{2n}$, $A_{1} \oplus C_{n}$) & $2\pi_{1}^{A_{1}} + \pi_{2}^{C_{n}} = \delta^{A_{1}} + \delta_{\text{short}}^{C_{n}}$ & $4n-4$ & $8n - 7$ & Yes \\ 
        31 & ($G_{2}$, $A_{1}$)& $10\pi_{1} = 5 \delta$ & 1 & 5 & \\ 
        32 & ($F_{4}$, $A_{1} \oplus G_{2}$)& $4\pi^{A_{1}}_{1} + \pi^{G_{2}}_{1}=2 \delta^{A_{1}} + \delta_{\text{short}}^{G_{2}}$ & 6 & 15 & \\ 
        33 & ($E_{6}$, $G_{2}$)& $\pi_{1}+\pi_{2}=\delta + \delta_{\text{short}}$ & 6 & 21 & \\ 
        34 & ($E_{6}$, $A_{2} \oplus G_{2}$)& $\pi_{1}^{A_{2}} + \pi^{A_{2}}_{2}+\pi^{G_{2}}_{1}=\delta^{A_{2}} + \delta_{\text{short}}^{G_{2}}$ & 8 & 21 & \\ 
        35 & ($E_{7}$, $A_{2}$)& $4\pi_{1} + 4 \pi_{2}=4 \delta$ & 3 & 33 & \\ 
        36 & ($E_{7}$, $C_{3} \oplus G_{2}$)& $\pi^{C_{3}}_{2} + \pi^{G_{2}}_{1}=\delta_{\text{short}}^{C_{3}} + \delta_{\text{short}}^{G_{2}}$ & 12 & 33 & \\[.1em] 
        37 & ($E_{7}$, $A_{1} \oplus F_{4}$)& $2 \pi^{A_{1}}_{1} + \pi^{F_{4}}_{1}=\delta^{A_{1}} + \delta_{\text{short}}^{F_{4}}$ & 16 & 33 & Yes \\[.1em] 
        38 & ($E_{8}$, $G_{2} \oplus F_{4}$)& $\pi^{G_{2}}_{1} + \pi^{F_{4}}_{1}=\delta_{\text{short}}^{G_{2}} + \delta_{\text{short}}^{F_{4}}$ & 20 & 57 & \\[.1em] \hline
    \end{tabular}%
    }
    \caption{Classification of non-symmetric isotropy irreducible pairs $(\gfr,\hfr)$.}
    \label{table: classification of g h}
\end{table}

\begin{table}
    \centering
    \hspace*{-.1\textwidth} 
    \resizebox{1.3\textwidth}{!}{
    \begin{tabular}{|c||c|c|c|c|}
    \hline
        \multirow{2}{*}{($\gfr$, $\hfr$)} & Root of $\gfr$ that is & \multirow{2}{*}{Marked Dynkin diagram of $O_{\mfr}$} & \multirow{2}{*}{$\dim O_{\mfr}$} & \multirow{2}{*}{$Z_{\mfr}^{\dim Z_{\mfr}}$} \\
        & the highest weight $\rho$ of $\mfr$&&&\\\hline \hline
        ($B_{l}$, $D_{p} \oplus B_{l-p}$) ($2 \le p \le l$) & $- \alpha_{p}$ & (if $p=2<l$) ${\dynkin[labels={-\delta}] A{x}} \otimes {\dynkin[labels={1}] A{x}} \otimes {\dynkin[labels={3,,l}] B{x.**}}$ & (if $p<l$) $2l-3$ & (if $p<l$) $Z_{\text{long}}^{4l-5}$ \\ 
        && (if $2<p<l$) ${\dynkin[backwards, labels={p-1, 2, -\delta, 1}, label directions={,right,,}] D{x.***}} \otimes {\dynkin[labels={p+1,,l}] B{x.**}}$ & & \\ 
        & & (if $p=l=2$) ${\dynkin[labels={-\delta}] A{x}} \otimes {\dynkin[labels={1}] A{x}}$ & (if $p=l$) $2l-2$ & (if $p=l$) $Z_{\text{short}}^{4l-3}$ \\
        & & (if $2<p=l$) {\dynkin[backwards, labels={l-1, 2, -\delta, 1}, label directions={,right,,}] D{x.***}} & & \\
        ($C_{l}$, $C_{p} \oplus C_{l-p}$) ($1 \le p \le l-1$) & $- \alpha_{p}$ & ${\dynkin[backwards, labels={p-1, 1,-\delta}] C{x.**}} \otimes {\dynkin[labels={p+1, l-1,l}] C{x.**}}$ & $2l-2$ & $Z_{\text{short}}^{4l-3}$ \\ 
        ($D_{l}$, $D_{p} \oplus D_{l-p}$) ($2 \le p \le l-2$) & $- \alpha_{p}$ & (if $p=2<l-2$) ${\dynkin[labels={-\delta}] A{x}} \otimes {\dynkin[labels={1}] A{x}} \otimes {\dynkin[labels={3,l-2,l-1,l}, label directions={,right,,}] D{x.***}}$ & $2l-4$ & $Z_{\text{long}}^{4l-7}$ \\ 
         & & (if $2<p<l-2$) ${\dynkin[backwards, labels={p-1,2,-\delta,1}, label directions={,right,,}] D{x.***}} \otimes {\dynkin[labels={p+1,l-2,l-1,l}, label directions={,right,,}] D{x.***}}$ & & \\ 
         & & (if $l=4$ and $p=2$) ${\dynkin[labels={-\delta}] A{x}} \otimes {\dynkin[labels={1}] A{x}} \otimes {\dynkin[labels={3}] A{x}} \otimes {\dynkin[labels={4}] A{x}}$ & & \\ 
         & & (if $2< p = l-2$) ${\dynkin[backwards, labels={l-3,2,-\delta,1},label directions={,right,,}] D{x.***}} \otimes {\dynkin[labels={l-1}] A{x}}\otimes {\dynkin[labels={l}] A{x}}$ & & \\ 
        ($G_{2}$, $A_{1} \oplus A_{1}$) & $- \alpha_{2}$ & ${\dynkin[labels={1}] A{x}} \otimes {\dynkin[labels={-\delta}] A{x}}$ & $2$ & $Z_{\text{long}}^{5}$ \\ 
        ($F_{4}$, $C_{3} \oplus C_{1}$) & $- \alpha_{4}$ & ${\dynkin[labels={1,2,3}] C{**x}} \otimes {\dynkin[labels={-\delta}] A{x}}$ & $7$ & $Z_{\text{long}}^{15}$ \\ 
        ($F_{4}$, $B_{4}$) & $- \alpha_{1}$ & ${\dynkin[labels={-\delta, 4,3,2}, backwards] B{***x}}$ & $10$ & $Z_{\text{short}}^{21}$ \\ 
        ($E_{6}$, $A_{5} \oplus A_{1}$) & $- \alpha_{2}$ & ${\dynkin[labels={1}] A{x}} \otimes {\dynkin[labels={-\delta, 6,3,4,5}] A{**x**}}$ & $10$ & $Z_{\text{long}}^{21}$ \\ 
         & $- \alpha_{4}$ & ${\dynkin[labels={1, 2,3,6,-\delta}] A{**x**}}  \otimes {\dynkin[labels={5}] A{x}}$ & & \\ 
         & $- \alpha_{6}$ & ${\dynkin[labels={1,2,3,4,5}] A{**x**}} \otimes {\dynkin[labels={-\delta}] A{x}}$ & & \\ 
        ($E_{7}$, $A_{7}$) & $- \alpha_{7}$ & ${\dynkin[labels={1, 2,3,4,5,6,-\delta}] A{***x***}}$ & $16$ & $Z_{\text{long}}^{33}$ \\ 
        ($E_{7}$, $D_{6} \oplus A_{1}$) & $-\alpha_{2}$ & ${\dynkin[labels={1}] A{x}} \otimes {\dynkin[backwards, labels={-\delta, 6, 5, 4, 3, 7}, label directions={, , , right, , }] D{****x*}}$ & $16$ & $Z_{\text{long}}^{33}$ \\ 
         & $- \alpha_{6}$ & ${\dynkin[labels={1, 2, 3, 4, 5, 7}, label directions={, , , right, , }] D{****x*}}  \otimes {\dynkin[labels={-\delta}] A{x}}$ & & \\ 
        ($E_{8}$, $D_{8}$) & $- \alpha_{7}$ & ${\dynkin[labels={-\delta,1,2,3,4,5,6,8}, label directions={,,,,,right,,}] D{******x*}}$ & $28$ & $Z_{\text{long}}^{57}$ \\ 
        ($E_{8}$, $E_{7} \oplus A_{1}$) & $- \alpha_{1}$ & ${\dynkin[labels={-\delta}] A{x}} \otimes {\dynkin[backwards, upside down, labels={7,8,6,5,4,3,2}] E{******x}}$& $28$ & $Z_{\text{long}}^{57}$ \\ 
        \hline
    \end{tabular}
    }
    \caption{Highest weight orbits for isotropy irreducible pairs $(\gfr,\, \hfr)$ of equal rank.}
    \label{table: isotropy irreducible pairs of equal rank}
\end{table}

\begin{table}
    \centering
    \resizebox{\textwidth}{!}{
    \begin{tabular}{|c||c|c|c|}
    \hline
        ($\gfr$, $\hfr$) & Highest weight $\rho$ of $\mfr$ & $\dim O_{\mfr}$ & $Z_{\mfr}^{\dim Z_{\mfr}}$ \\ \hline \hline
        ($A_{l} \oplus A_{l},\, \text{diag}(A_{l})$) ($l \ge 1$) & $\delta$ & $2l-1$ & $\PP(\Ocal_{\text{min}} \oplus \Ocal_{\text{min}})^{4l-1}$ \\ 
        ($\mathfrak{so}(l) \oplus \mathfrak{so}(l),\, \text{diag}(\mathfrak{so}(l))$) ($l\ge 7$) & $\delta$ & $2l-7$ & $\PP(\Ocal_{\text{min}} \oplus \Ocal_{\text{min}})^{4l-13}$ \\ 
        ($C_{l} \oplus C_{l},\, \text{diag}(C_{l})$) ($l \ge 2$) & $\delta$ & $2l-1$ & $\PP(\Ocal_{\text{min}} \oplus \Ocal_{\text{min}})^{4l-1}$ \\ 
        ($G_{2} \oplus G_{2},\, \text{diag}(G_{2})$) & $\delta$ & $5$ & $\PP(\Ocal_{\text{min}} \oplus \Ocal_{\text{min}})^{11}$\\ 
        ($F_{4} \oplus F_{4},\, \text{diag}(F_{4})$) & $\delta$ & $15$ & $\PP(\Ocal_{\text{min}} \oplus \Ocal_{\text{min}})^{31}$ \\ 
        ($E_{6} \oplus E_{6},\, \text{diag}(E_{6})$) & $\delta$ & $21$ & $\PP(\Ocal_{\text{min}} \oplus \Ocal_{\text{min}})^{43}$ \\ 
        ($E_{7} \oplus E_{7},\, \text{diag}(E_{7})$) & $\delta$ & $33$ & $\PP(\Ocal_{\text{min}} \oplus \Ocal_{\text{min}})^{67}$ \\ 
        ($E_{8} \oplus E_{8},\, \text{diag}(E_{8})$) & $\delta$ & $57$ & $\PP(\Ocal_{\text{min}} \oplus \Ocal_{\text{min}})^{115}$ \\ 
        ($A_{l-1}$, $\mathfrak{so}(l)$) ($l \ge 3$) & (if $l=3$) $4 \pi_{1}^{A_{1}}$ & $l-2$ & $Z_{\text{long}}^{2l-3}$ \\ 
        & (if $l=4$) $2 \pi_{1}^{A_{1}'} + 2 \pi_{1}^{A_{1}''}$ & & \\ 
        & (if $l\ge 5$) $2 \pi_{1}^{\mathfrak{so}(l)}$ & & \\ 
        ($A_{2l-1}$, $C_{l}$) ($l \ge 2$) & $\pi_{2} = \delta_{\text{short}}$ & $4l-5$ & $Z_{[2^{2},\, 1^{2l-4}]}^{8l-9}$ \\ 
        ($D_{p+q+1}$, $B_{p} \oplus B_{q}$) ($p+q\ge 2$, $p \ge q \ge 0$) & (if $q > 0$) $\pi_{1}^{B_{p}} + \pi_{1}^{B_{q}}$ & (if $q > 0$) $2p+2q-2$ & (if $q > 0$) $Z_{\text{long}}^{4p+4q-3}$ \\ 
        & (if $q = 0$) $\pi_{1}^{B_{p}} = \delta_{\text{short}}$ & (if $q = 0$) $2p-1$ & (if $q=0$) $Z_{[3,\, 1^{2p-1}]}^{4p-1}$ \\ 
        ($E_{6}$, $F_{4}$) & $\pi_{1} = \delta_{\text{short}}$ & $15$ & $Z_{2 A_{1}}^{31}$ \\ 
        ($E_{6}$, $C_{4}$) & $\pi_{4}$ & $10$ & $Z_{\text{long}}^{21}$ \\ \hline
    \end{tabular}
    }
    \caption{Highest weight orbits for symmetric isotropy irreducible pairs $(\gfr,\, \hfr)$ of different rank.}
    \label{table: symmetric isotropy irreducible pairs of different rank with g simple}
\end{table}

\begin{table}
    \centering
    \begin{tabular}{|c||c|c|c|}
    \hline
        $S_{\text{ad}} / P$ & Marked Dynkin diagram of $O$ & $\dim O$ & $\dim Z_{\text{long}}$ \\ \hline \hline
        $A_{l}/P_{p} (\simeq \text{Gr}(p,\, \CC^{l+1}))$ ($1 \le p \le l$) & ${\dynkin[labels={1,2,p-1}] A{x*.*}} \otimes {\dynkin[labels={p+1,,l}] A{*.*x}}$ & $l-1$ & $2l-1$ \\ 
        $B_{l}/P_{1} (\simeq \QQ^{2l-1})$ ($l \ge 3$) & {\dynkin[labels={2,3,,l}] B{x*.**}} & $2l-3$ & $4l-5$ \\ 
        $C_{l}/P_{l}(\simeq\text{LG}(l,\, \CC^{2l}))$ ($l \ge 2$) & {\dynkin[labels={1,2,l-1}] A{x*.*}} & $l-1$ & $2l-1$ \\ 
        $D_{l}/P_{1}(\simeq\QQ^{2l-2})$ ($l \ge 4$) & {\dynkin[labels={2,3,l-2,l-1,l}, label directions={,,right,,}] D{x*.***}} & $2l-4$ & $4l-7$ \\ 
        $D_{l} / P_{p} (\simeq \SS_{l})$ ($l \ge 4$, $p=l-1,\,l$) & $\underbrace{\dynkin A{*x*.**}}_{(l-1) \text{ nodes}}$ & $2l-4$ & $4l-7$ \\ 
        $E_{6}/P_{p}(\simeq \mathbb{OP}^{2})$ ($p=1,\,5$) & {\dynkin D{****x}} & $10$ & $21$ \\ 
        $E_{7}/P_{1}$ & {\dynkin[upside down, labels={2,7,3,4,5,6}] E{*****x}} & $16$ & $33$ \\ \hline
    \end{tabular}
    \caption{Unique closed $P$-orbits $O$ associated to irreducible Hermitian symmetric spaces $S_{\text{ad}} / P$.}
    \label{table: Legendrian asso to IHSS}
\end{table}

\end{document}